\documentclass[12pt,english,draft]{amsart}
\usepackage{mathptmx}
\usepackage[T1]{fontenc}
\usepackage[latin9]{inputenc}
\usepackage{geometry}
\usepackage{amssymb}
\usepackage{amsmath}
\usepackage{bbm}
\usepackage{enumerate}
\usepackage{amsthm}
\theoremstyle{plain}
\newtheorem{thm}{Theorem}[section]
  \theoremstyle{plain}
  \newtheorem{cor}[thm]{Corollary}
  \theoremstyle{plain}
  \newtheorem{prop}[thm]{Proposition}
  \theoremstyle{plain}
  \newtheorem*{thm*}{Theorem}
  \theoremstyle{remark}
  \newtheorem{rem}[thm]{Remark}
 \theoremstyle{definition}
  \newtheorem{example}[thm]{Example}
  \theoremstyle{definition}
  \newtheorem{defn}[thm]{Definition}
  \theoremstyle{plain}
  \newtheorem{lemma}[thm]{Lemma}
  \newtheorem*{mainprop*}{Proposition 1.4}  % To keep it with the same number as in the introduction.
  \newtheorem*{mainpropa*}{Proposition 1.3} 

  \newtheorem*{mainthm*}{Theorem 1.1}  % To keep it with the same number as in the introduction.
\numberwithin{equation}{section}
\numberwithin{figure}{section}
\providecommand{\SL}[1]{\mathop{\rm SL_2}(\mathbb{#1})}
\providecommand{\PSL}[1]{\mathop{\rm PSL_2}(\mathbb{#1})}
\providecommand{\GL}[1]{\mathop{\rm GL_2}(\mathbb{#1})}
\providecommand{\PGL}[1]{\mathop{\rm PGL_2}(\mathbb{#1})}

\providecommand{\GLF}[1]{\mathop{\rm GL_2}(\mathbb{F}_{#1})}
\providecommand{\PGLF}[1]{\mathop{\rm PGL_2}(\mathbb{F}_{#1})}
\providecommand{\GLZhat}{\mathop{\rm GL_2}(\hat{\mathbb{Z}})}
\providecommand{\GLRplus}{\mathop{\rm GL_2}^{+}({\mathbb{R}})}
\providecommand{\GLZp}{\mathop{\rm GL_2}({\mathbb{Z}_p})}

\providecommand{\SLZ}{\SL{Z}}
\providecommand{\SLR}{\SL{R}}

\providecommand{\PSLZ}{\PSL{Z}}
\providecommand{\PGLZ}{\PGL{Z}}
\providecommand{\PSLR}{\PSL{R}}
\providecommand{\PGLR}{\PGL{R}}
\providecommand{\PSLQ}{\PSL{Q}}
\providecommand{\PGLQ}{\PGL{Q}}

\providecommand{\Id}{\mathop{\rm Id}}
\providecommand{\Norm}{\mathop{\rm Norm}}
\renewcommand{\H}{\mathbb{H}}
\newcommand{\R}{\mathbb{R}}
\newcommand{\C}{\mathbb{C}}
\newcommand{\N}{\mathbb{N}}
\newcommand{\Z}{\mathbb{Z}}
\newcommand{\Q}{\mathbb{Q}}
\renewcommand{\AA}{\mathbb{A}}

\renewcommand{\S}{{S}}

\newcommand{\ip}[1]{\left< \smash #1 \right>}
\newcommand{\St}{T^{1/3}}
\newcommand{\Stinv}{T^{-1/3}}
\newcommand{\VV}{\mathcal{V}}
\newcommand{\A}{V}     % \A_d is the dilation: z->dz
\newcommand{\sgn}{\mathrm{sgn}}
\newcommand{\std}{\mathrm{std}}

\newcommand{\Eins}{\mathbbm{1}}
\newcommand{\n}{\mathrm{n}} %genuinely newforms
\newcommand{\gn}{\mathrm{gn}} %generalized oldforms
\renewcommand{\o}{\mathrm{o}} %genuinely newforms
\renewcommand{\t}{\mathrm{t}} % twists
\newcommand{\go}{\mathrm{go}} %generalized oldforms
\newcommand{\mn}{\mathcal{M}^{\n}} %newforms
\newcommand{\mo}{\mathcal{M}^{\o}} %oldforms
\newcommand{\mgn}{\mathcal{M}^{\gn}} %genuinely newforms
\newcommand{\mgo}{\mathcal{M}^{\go}} %generalized oldforms

\makeatother

\title{Newforms and Spectral Multiplicity for $\Gamma_{0}(9)$}

\author{Fredrik Str{\"o}mberg}
\address{
   Fredrik Str{\"o}mberg\\
   AG Algebra, FB Mathematik \\
   Technische Universit{\"a}t Darmstadt \\
   Schlossgartenstra{\ss}e 7\\
   64289 Darmstadt \\
   Germany
  }
 \email{stroemberg@mathematik.tu-darmstadt.de}

\begin{document}
 
\maketitle

\begin{abstract}
The goal of this paper is to explain certain experimentally observed
properties of the (cuspidal) spectrum and its associated automorphic
forms (Maass waveforms) on the congruence subgroup $\Gamma_{0}(9)$.
The first property is that the spectrum possesses multiplicities in
the so-called new part, where it was previously believed to be simple.
The second property is that the spectrum does not contain any ``genuinely
new'' eigenvalues, in the sense that all eigenvalues of $\Gamma_{0}(9)$
appear in the spectrum of some congruence subgroup of lower level. 

The main theorem in this paper gives a precise decomposition of the
spectrum of $\Gamma_{0}(9)$ and in particular we show
that the genuinely new part is empty. We also prove that there exist
an infinite number of eigenvalues of $\Gamma_{0}(9)$ where
the corresponding eigenspace is of dimension at least two and has
a basis of pairs of Hecke-Maass newforms which are related to each
other by a character twist. These forms are non-holomorphic analogues
of modular forms with inner twists and also provide explicit (affirmative)
examples of a conjecture stating that if the Hecke eigenvalues of
two ``generic'' Maass newforms coincide on a set of primes of
density $\frac{1}{2}$ then they have to be related by a character
twist.
\end{abstract}

%\maketitle
%\tableofcontents

\section{Introduction}
\label{sec:intro}

\subsection{A computational approach to modular forms}
Since the works of Gau{\ss} and Riemann experiments have always been
an important part of research in analytic number theory. However,
due to both an increasing complexity of the objects studied as well
as a change in the point of view of mathematical research, the experimental
component did not play an as important role (relatively speaking)
in the number-theoretical development in the beginning and mid 20th
century (Ramanujan being a brilliant exception). Since then, beginning
with the advent of the modern computer, the situation has gradually
changed in favor of experiments. This is especially true in the area
of L-functions and modular forms, where experiments play a prominent
role today. This situation has its roots in both the discovery of
effective computational methods and the development of cheap and fast
processors. The latter is essentially giving the average researcher
access, through a common desktop computer, to more raw computational
power than that of a supercomputer some twenty years ago. 

For holomorphic modular forms the computations are greatly assisted
by using associated algebraic and geometric objects, essentially enabling
all necessary computations to be performed using integer arithmetic
over number fields. For non-holomorphic modular forms, Maass waveforms,
there is in general no algebraic or geometric theory to fall back
on and it is necessary to rely on floating point calculations. The only instances
of Maass waveforms where there exist explicit formulas, or even explicit
information about their field of definition, are forms of CM-type
coming from lifts of Hecke Gr{\"o}{\ss}encharakters (see e.g.~\cite{maass:49,bump,CM}).
It is believed that for a generic Maass waveform (not of CM-type)
the Laplace eigenvalue as well as a certain (infinite) subset of Fourier
coefficients are all algebraically independent transcendental numbers. 
This belief is far from proven but it is supported through extensive numerical experiments 
in ~e.g.~\cite{andreas:effective_comp}. 
At this point it is worth mentioning that the arithmetic properties of {\em Harmonic weak} Maass forms (which we do not treat here) are
completely different. See e.g. \cite{BO2007,BO2010,BruO}.

In view of this (apparent) non-arithmeticity it is not surprising that the experimental and computational
component of the fundamental research is even stronger in the area
of non-holomorphic than holomorphic modular (and automorphic) forms.
The very first computations of (non CM-type) Maass waveforms were
performed in the 70's. For an overview of the early history and a collection of 
tables see \cite[Appendix C]{hejhal:lnm1001}. In the 80's and 90's
more efficient algorithms were developed by Stark and Hejhal (both
separately and in collaboration). See e.g. \cite{MR803370} and \cite{hejhal:99_eigenf}.
The algorithm developed by Hejhal in \cite{hejhal:99_eigenf}, based
on so-called implicit automorphy, has been generalized in various
directions by the author \cite{stromberg:thesis,stromberg:weight},
Avelin \cite{avelin_thesis:07,avelin_eisenstein,helen:deform_published}
and Then \cite{then:02,MR2261095}.

\subsection{Experimental observations of the spectrum}

One of the features observed in the early computations of Maass cusp
forms for the modular group, $\SLZ$,
was that the spectrum seemed to be simple (cf. e.g.~\cite{MR0314733:Cartier,MR1736930,hejhal:92,steil:evs}).
This has later been verified (numerically) for a large number of eigenvalues
for the modular group (cf.~e.g~\cite{then:02}), as well as for
certain parts of the spectrum corresponding to congruence subgroups,
the so called ``new'' part (in the sense of Atkin-Lehner \cite{atkinlehner}).
Based on these computations and ``general principles'' (e.g. the Wigner-von
Neumann Theorem \cite{wigner-vn}) it has since been widely believed
that the new part of the spectrum on congruence subgroups is simple. 

During numerical investigations in connection with \cite{stromberg:thesis,stromberg:weight}
two cases contradicting this belief were discovered. The first such
case was for groups $\Gamma_{0}(p)$, with $p$ prime,
and Nebentypus given by the quadratic residue symbol modulo $p$.
In this case it is easy to see that $\varphi$ and $\overline{\varphi}$
(where $\overline{\varphi}$ is the complex conjugate of $\varphi$)
have the same transformation properties, but are not necessarily identical.
In fact, in this case $\varphi=\overline{\varphi}$ holds only for
CM-type forms and these correspond to a subset of the spectrum of
density zero. See e.g.~\cite[Rem.~1.2.3]{stromberg:thesis}. An
analogous but more intricate argument provides similar examples of
multiplicity for composite levels and certain types of non-trivial
characters, cf.~e.g.~Booker and Str\"ombergsson \cite[Sec.\ 3.4]{andreas:small_ev}. 

The purpose of this paper is to explain another kind of multiplicity
occurring for $\Gamma_{0}(9)$ with trivial Nebentypus.
The Laplace spectrum for this group was studied (numerically) independently
by the author (in connection with \cite{stromberg:thesis}) and by
D. Farmer and S. Lemurell (in connection with \cite{farmer-lemurell}). 

\begin{table}[!ht]
\caption{Eigenvalues $\lambda$ of $\Gamma_{0}(9)$ with $\lambda=\frac{1}{4}+R^{2}$ and $R<10$}
\begin{tabular}{c|c}

\begin{tabular}[t]{l|l}
\multicolumn{1}{c}{$R$  (odd)}& {new on}\\
% ($\varphi_{|J}=-\varphi$)& &
% $\omega_{9}$ & 

\noalign{\smallskip}\hline\noalign{\smallskip}
3.5360020929376	& {$\Gamma^3$} \\
4.3880535632221	& {$\Gamma_0(3)$} \\
& \\
5.5040566796766	& {$\Gamma^3$} \\
6.1205755330872	& {$\Gamma_0(3)$} \\
6.6465813556094	& {$\Gamma^3$} \\
6.7574152777543	& {$\Gamma_0(3)$} \\
7.4317991717218	& {$\Gamma^3$} \\
& \\
7.7581331950210	& {$\Gamma_0(3)$} \\
8.193035931685	& {$\Gamma_0(3)$} \\
8.698342956468	& {$\Gamma^3$} \\
& \\
9.0800693497016	& {$\Gamma^3$} \\
9.2923793282248 & {$\Gamma_0(3)$} \\
9.5336952613536	& {$\Gamma_0(1)$} \\

\end{tabular} &
\begin{tabular}[t]{l|l}
\multicolumn{1}{c}{$R$  (even)} & {new on}\\
%($\varphi_{|J}=\varphi$)&& 
\noalign{\smallskip}\hline\noalign{\smallskip}
3.5360020929376	& {$\Gamma^3$} \\
& \\
5.0987419087295	& {$\Gamma_0(3)$} \\
5.5040566796765	& {$\Gamma^3$} \\
& \\
6.6465813556097	& {$\Gamma^3$} \\
& \\
7.4317991717219	& {$\Gamma^3$} \\
8.038861203863	&  {$\Gamma_0(3)$} \\
& \\
& \\
8.6983429564685	& {$\Gamma^3$} \\
8.7782823935563	&  {$\Gamma_0(3)$} \\
9.080069349694	& {$\Gamma^3$} \\
& \\
& \\
9.74374939916	& {$\Gamma_0(3)$} \\

\end{tabular} 
\end{tabular} 
\end{table}

The first few computed spectral parameters for $\Gamma_{0}(9)$,
together with some additional information, are listed in Table 1.
For the purpose of numerical computations it is advantageous to desymmetrize
the spectrum and split it into its ``odd'' and ``even'' parts,
where an eigenvalue is said to be odd or even if the corresponding
eigenfunction is odd or even, respectively, with respect to the reflection
in the imaginary axis, $z\mapsto-\overline{z}$. From the numerical
data alone it is clear that
something interesting is happening here, since there are eigenvalues
showing up as both odd and even. Although eigenvalues associated to
oldforms always appear with multiplicity at least two, they can never
appear in both the odd and the even part of the spectrum simultaneously.
A closer inspection of the spectrum of the modular group and of $\Gamma_{0}(3)$
showed that these (at that point unexplained) even and odd multiplicities
were all associated to newforms, making the situation even more intriguing.
At this point it should also be mentioned that the pattern shown in
Table 1 seems to continue, at least up to $R=50$. Based on the numerical
data alone, we were tempted to {\em conjecture} that the entire
part of the spectrum not associated with $\PSLZ$ or $\Gamma_{0}(3)$
(that is oldforms and twists) comes with multiplicity greater than
or equal to two. The real key to proving this conjecture was the discovery
that all the new eigenvalues also appeared in the spectrum of another
congruence group: $\Gamma^{3}$ (and also there with multiplicities).
It is therefore equally tempting to conjecture that all of the spectrum
of $\Gamma_{0}(9)$ is accounted for by looking at the
modular group, $\Gamma_{0}(3)$ and $\Gamma^{3}$.

\subsection{Formulation and discussion of the main theorem}

In the course of proving the conjectures mentioned above (cf.~Corollary \ref{cor:n-genuinely_new_maass_forms} and Proposition \ref{prop:A-positive-proportion-has-mult}), we arrived
at Theorem \ref{thm:main_thm} which is presented below and is given in more detail in Section \ref{sec:The-Main-Theorem}. 
This theorem gives a complete description of the space of Maass waveforms on $\Gamma_{0}(9)$
in terms of forms on congruence subgroups of lower level. 
\begin{thm}
\label{thm:main_thm}The space of Maass waveforms on $\Gamma_{0}(9)$
can be decomposed as a direct sum of spaces of forms of the following
four types: oldforms, twists of forms on the modular group,
twists of forms on $\Gamma_{0}(3)$ and forms related to newforms
on the group $\Gamma^{3}$. Furthermore, the last constituent is non-empty
for an infinite number of eigenvalues. 
\end{thm}
The precise meaning of ``oldforms'' and ``twists'' in this
situation will be explained in Sections \ref{sub:Maass-waveforms}
and \ref{sub:Lifts-and-twists}. 
We say that a Maass waveform on a
congruence subgroup, which is neither a twist of, nor associated to, a Maass
waveform on a group of lower level is {\em genuinely new}. 
This notion will be discussed further in Section \ref{sec:The-Main-Theorem} and we can formulate the following consequence of Theorem \ref{thm:main_thm}.
%is the following corollary.
\begin{cor}
\label{cor:n-genuinely_new_maass_forms}There are no genuinely new
Maass waveforms on $\Gamma_{0}(9)$. 
\end{cor}
By exploiting properties of the group $\Gamma^3$, in particular that it has two non-commuting involutions $z\mapsto z+1$ and $z\mapsto -\bar{z}$,
we will also show the following proposition in Section \ref{sub:Lifts-from-}.
\begin{prop}
\label{prop:A-positive-proportion-has-mult}
Two fifth of the new part of the spectrum of $\Gamma_{0}(9)$ has multiplicity at least two.
%the new part of the spectrum of $\Gamma_{0}(9)$ has multiplicity
%at least two.
\end{prop}
To prove the main theorem we calculate all different types of maps taking 
forms on $\PSLZ$, $\Gamma_{0}(3)$ and
$\Gamma^{3}$ to forms on $\Gamma_{0}(9)$. We then show
that the corresponding images inside the space of Maass waveforms on $\Gamma_0(9)$, 
are all disjoint (and in fact orthogonal). To show that the complement
of the sum of these images is empty we write down the Selberg trace
formula explicitly for all the groups involved. Using an inclusion--exclusion argument and subtracting off
all contributions from forms associated to lower level subgroups %from that of $\Gamma_{0}(9)$
we are then able to show that the the remaining part vanishes completely.

Note that, in order to prove a weaker version of Proposition \ref{prop:A-positive-proportion-has-mult} 
with ``A positive proportion`` replaced by ``An infinite number''  we would not need to use the full Selberg trace
formula; using Weyl's law (which essentially amounts to the main term) would
be enough. 

%FS: Added
In Section \ref{sec:A-representation-theoretic-inter} we give a brief
discussion on a representation-theoretical interpretation of the main theorem.
We also sketch an alternative proof of Corollary \ref{cor:n-genuinely_new_maass_forms} using elementary
representation theory of finite groups, together with the theory of automorphic representations of adele groups.

\subsection{Applications}
It turns out that the construction we use to prove the existence of
forms with multiple Laplace eigenvalues is compatible with the Hecke
theory for Maass waveforms on $\Gamma_{0}(9)$ (see Sections
\ref{sub:Maass-waveforms} and \ref{sub:Hecke-operators-onGamma09}).
In Section \ref{sub:Hecke-operators-onGamma09} we will prove the following
Proposition.
\begin{prop}
\label{prop:Hecke-pair}
There exist an infinite number of pairs $\left\{ F^{+},F^{-}\right\}$
of Maass newforms on $\Gamma_{0}(9)$ with the property
that the Hecke eigenvalues $c^{+}(n)$ and $c^{-}(n)$ of $F^{+}$ and $F^{-}$ are related through 
$c^{-}(n)=\left(\frac{n}{3}\right)c^{+}(n)$
for all $n$ relatively prime to $3$. 
\end{prop}
A pair of functions as in Proposition \ref{prop:Hecke-pair} have
Hecke eigenvalues which coincide for a set of primes of density $\frac{1}{2}$,
but unless $c^{+}(n)=c^{-}(n)=0$ for all $n\equiv-1\mod3$
they are not identical (this is essentially the condition of being
a CM-type form). It therefore provides an example of a conjecture
of Rajan \cite{MR1606395,MR1707005,MR1994478} (here formulated in
the setting of Maass waveforms): if two generic (non CM-type) cuspidal
Maass newforms have Hecke eigenvalues agreeing for a set of primes
of positive density then one of them must be a twist of the other
by a Dirichlet character. The following variant of this conjecture
was proven by Ramakrishnan \cite[Cor.\ 4.1.3]{MR1792292}: if two
cuspidal Maass newforms $f$ and $g$ have Hecke eigenvalues whose
squares are equal then they are in fact related through a twist by
a Dirichlet character. 

The following multiplicity one theorem was proven by Ramakrishnan
\cite{MR1253208}: 
\begin{thm}
[Ramakrishnan]\label{thm:thm.1_over_8}Let $f$ and $g$ be two Maass
cusp forms with Hecke eigenvalues $a_{f}(p)$ and $a_{g}(p)$.
If the proportion of primes $p$ for which $a_{f}(p)\ne a_{g}(p)$
is less than $\frac{1}{8}$ then $f=g$. 
\end{thm}
By examples constructed by Serre and others (see e.g.~\cite{MR1265561})
it is known that in the general formulation above, the number $\frac{1}{8}$
is sharp. However, the constructed examples are all of special types
(corresponding to CM-type forms), and by the analogous results in
the $l$-adic setting by Rajan \cite{MR1606395} it is in fact expected
(cf.~e.g.~\cite[p.\ 189]{MR1994478}) that the number $\frac{1}{8}$
can be replaced by $\frac{1}{2}$ under the assumption that $f$ and
$g$ are of general (non-CM) type. The existence of the forms in Proposition
\ref{prop:Hecke-pair} proves the following corollary.
\begin{cor}
\label{cor:sharpness_of_1_over_2}It is not possible to replace $\frac{1}{8}$
in the above theorem with any number greater than $\frac{1}{2}$ even
under the assumption that $f$ and $g$ are of general type. 
\end{cor}
By using the same methods as
presented in this paper it is possible to prove the analogue of Theorem \ref{thm:main_thm}
for $\mathcal{M}_{k}(9)$, the space of holomorphic modular
forms of weight $k$ on $\Gamma_{0}(9)$ (and similarly
for the space of cusp forms $\mathcal{S}_{k}(9)$). The
only modification to the proof is that we do not need to use the Selberg
 trace formula but can instead use a standard dimension formula for spaces of holomorphic modular forms, cf. e.g.~\cite[Thm.~2.5.2]{miyake}. 
\begin{thm}
The decomposition given in Theorem \ref{thm:main_thm} holds also
for holomorphic modular forms.
\end{thm}
For the benefit of the reader we also formulate the following immediate
consequence of the above theorem.
\begin{cor}
\label{thm:main_cor_for_modular_forms}There are no genuinely new
holomorphic modular forms on $\Gamma_{0}(9)$. 
\end{cor}
We should also mention that there is one part of this paper which
does not go through directly to the setting of holomorphic modular
forms. Since the reflection $z\mapsto-\overline{z}$ is anti-holomorphic
we can not use this operator to force multiplicity. Without this tool
we are not able to prove that the operator $\St$ acting on the space
of newforms in $\mathcal{M}_{k}(9)$ possesses non-zero
eigenfunctions with eigenvalue $\zeta_{3}$ as well as $\zeta_{3}^{2}$.
We can therefore only say that there exists an infinite sequence of
holomorphic newforms on $\Gamma_{0}(9)$ which are either
invariant under twisting by $\chi=\left(\frac{\cdot}{3}\right)$ or
appear in pairs $f,g$ with $f_{\chi}=g$. The first case corresponds to forms of CM type and that the second
case corresponds to forms with inner twists. For a precise definition
and results related to inner twists see e.g.~\cite{MR0453647,MR594532,MR617867,MR2477510}.
For the sake of clarity we give a simple example of the situation
for holomorphic modular forms. 
\begin{example}
Consider the space of cusp forms of weight $12$ on $\Gamma_{0}(9)$.
Using e.g.~Sage \cite{sage} we see that $\mathcal{S}_{12}^{\n}(9)=\left\{ f,g,h\right\} $
with the following $q$-expansions ($q=e^{2\pi i\tau}$ for $\tau\in\H$):
\begin{align*}
f & (\tau)=q+24q^{2}-1472q^{4}-4830q^{5}-16744q^{7}-84480q^{8}-115920q^{10}-\cdots,\\
g & (\tau)=q-78q^{2}+4036q^{4}+5370q^{5}-27760q^{7}-155064q^{8}-418860q^{10}-\cdots,\\
h & (\tau)=q+xq^{2}+472q^{4}+224xq^{5}+58100q^{7}-1576xq^{8}+564480q^{10}-\cdots\end{align*}
where $x^{2}-2520=0$. The newform $h$ has two embeddings, $h^{\pm}$,
given by $\sigma^{\pm}:a+b\sqrt{2520}\mapsto a\pm b\sqrt{2520}$.
The embeddings act on the Fourier coefficients in the usual way, i.e.
if $h=\sum a_{n}q^{n}$ then $h_{\sigma^{\pm}}=\sum\sigma^{\pm}(a_{n})q^{n}.$
These expansions can be extended as far as necessary to prove the
following claims (in fact $13$ terms are enough): $f$ is a twist
of $\Delta\in\mathcal{S}_{12}(1)$ and $g$ is a twist
of the unique newform in $\mathcal{S}_{12}(3)$ (both are
twists by $\chi_{3}$). The factor $x$ is only present in the coefficients
$a_{n}$ of $h$ with $n\equiv2\mod3$ and hence $\sigma^{-}(a_{n})=\chi_{3}(a_{n})$,
i.e. $h_{\chi_{3}}^{+}=h^{-}$ meaning that $h$ has an inner twist
with $\chi_{3}$.
\end{example}

\subsection{Road map}

In Section \ref{sec:Background-Material} we will review the necessary
background material and definitions. First we will recall basic facts
about Fuchsian groups. In particular we will give detailed information
about those congruence subgroups which play a role in Theorem \ref{thm:main_thm}.
We will also introduce Maass waveforms, Hecke operators and the theory
of Maass newforms. The various subspaces of Maass waveforms mentioned in the Main Theorem
will then be properly defined. The section on background material
ends with a precise formulation of the Selberg trace formula for the
full modular group. 

The next section is devoted to a more detailed study of Maass waveforms
on the group $\Gamma^{3}$. In particular we will give a complete
description of the forms in this space which are associated to the modular group and
the maps taking forms on $\Gamma^{3}$ to forms on $\Gamma_{0}(9)$. In Section
\ref{sec:The-Main-Theorem} we give a precise definition of the genuinely
new forms and a formulation and proof of the Main Theorem. 

The most technical section of the paper is Section \ref{sub:The-Selberg-trace-subgroups}
which gives the details of the Selberg trace formula for all the subgroups
we need. The section concludes with the proof of the key result, that
the space of genuinely new forms on $\Gamma_{0}(9)$ is
empty. 

A slightly modified version of the Hecke theory for $\Gamma_{0}(9)$
is introduced in Section \ref{sub:Hecke-operators-onGamma09}, for
the purpose of demonstrating that the construction of multiplicities
is compatible with Hecke operators (Proposition \ref{prop:Hecke-pair}).
%%% FS: CHANGED TO

In the bulk of this paper we take a classical approach to Maass waveforms, that is, we treat them as functions on the complex upper half-plane.
An alternative approach is given in terms of automorphic representations. In Section \ref{sec:Some-Perspectives-from} we 
give an interpretation of our results in this setting. 
The last section contains a discussion on possible extensions of the results in this paper. 

\section{Background material}
\label{sec:Background-Material}
We will start with a brief recollection of basic facts concerning
Fuchsian groups. For a more thorough treatment, the reader may consult
any of a number of textbooks, for example \cite{MR0164033,MR698777,katok}
or the first chapter of \cite{miyake}, to name a few. 

\subsection{Hyperbolic geometry and Fuchsian groups}
\label{sub:hypgeom}
Let $\H=\left\{ z=x+iy\,|\, y>0\right\} $ be the hyperbolic upper
half-plane equipped with metric and area measure $ds=y^{-1}\left|dz\right|$
and $d\mu=y^{-2}dxdy$, respectively. The boundary of $\H$ is given
by $\partial\H=\R\cup\left\{ \infty\right\} $ and a geodesic
on $\H$ is either a half-circle perpendicular to the real axis or
a straight line parallel to the imaginary axis. The group of orientation
preserving isometries of $\H$ is $\PSLR\simeq\SLR/\left\{ \pm\Eins_{2}\right\} $
acting on $\H$ by M{\"o}bius transformations $\gamma:z\mapsto\gamma z=\frac{az+b}{cz+d}$.
Here $\Eins_{2}$ is the two-by-two identity matrix. Composition of
maps correspond to matrix multiplication and we will usually write
elements of $\PSLR$ as matrices in $\SLR$. For notational clarity
we will, whenever possible, use matrices with integer entries. For example, 
we will often use the following two maps in the remaining part of the paper: 
\begin{align}
\omega_{d}:&z\mapsto-\frac{1}{dz} \quad \text{and}\\
\A_{d}:&z\mapsto d\,z.
\end{align}
Instead of matrices with determinant one, we will use the matrices  
$\left(\begin{smallmatrix}0 & -1\\
d & 0\end{smallmatrix}\right)$ and 
$\left(\begin{smallmatrix} d & 0\\  0 & 1 \end{smallmatrix}\right)
$ to represent these maps.
%is instead represented by $\left(\begin{smallmatrix}0 & -1\\
%3 & 0\end{smallmatrix}\right)$. 
Since we will never deal with actual elements of matrix groups this
should not cause any confusion, as long as it is understood that $\left(\begin{smallmatrix}a & b\\
c & d\end{smallmatrix}\right)$ and $\left(\begin{smallmatrix}\lambda a & \lambda b\\
\lambda c & \lambda d\end{smallmatrix}\right)$ are equal if $\lambda$ is a non-zero real number. 

If $\varphi:\H\rightarrow\C$ is a function then $\PSLR$
acts on $\varphi$ by the (weight zero) \emph{slash-action}: $\gamma\mapsto\varphi_{|\gamma}$,
where $\varphi_{|\gamma}(z)=\varphi(\gamma z)$.
The Laplace-Beltrami operator on $\H$ is given by $\Delta=-y^{2}\left(\frac{\partial^{2}}{\partial x^{2}}+\frac{\partial^{2}}{\partial y^{2}}\right)$
and since it only depends on the metric it commutes with the
action of $\PSLR$ (this can also be checked explicitly). Analogous
to the above action we can also define an action of $\PGLR$ on $\H$,
setting $\gamma z=\frac{a\overline{z}+b}{c\overline{z}+d}$ for $\gamma=\left(\begin{smallmatrix}a & b\\
c & d\end{smallmatrix}\right)$ and $ad-bc=-1$. In practice, the only map of this form we will need
is the reflection in the imaginary axis: $J:z\mapsto-\overline{z},$
which can be represented by $\left(\begin{smallmatrix}-1 & 0\\
0 & 1\end{smallmatrix}\right)$. 

A \emph{Fuchsian group} is a discrete subgroup of $\PSLR$, by which
we mean that if $M_{1},M_{2},\ldots$ is a sequence of elements in
the group and $M_{n}\rightarrow\Id$, the identity in $\PSLR,$
as $n\rightarrow\infty$, then there is an $M>0$ such that $M_{m}=\Id$
for all $m\ge M$. The orbit space, or orbifold, given by the quotient
of the upper half-plane by a Fuchsian group $\Gamma$, written $\Gamma\backslash\H:=\left\{ \Gamma z\,:\, z\in\H\right\} $
can be given a complex analytic structure which might be smooth or
contain singularities, either in the form of corners (elliptic points),
or of the form of punctures (parabolic points or cusps). It is common
to represent $\Gamma\backslash\H$ in the upper half-plane by a fundamental
domain, $\mathcal{F}_{\Gamma}$, which, for simplicity, we will always
assume to be closed (i.e. parts of the boundary are counted twice), convex and bounded by a finite number of geodesic
arcs (that is, we only consider Fuchsian groups which are geometrically
finite). The elliptic points and cusps correspond to points $z$ with
non-trivial stabilizer subgroup: $\Gamma_{z}:=\left\{ \gamma\in\Gamma\,|\,\gamma z=z\right\} $.
Such points necessarily appear on the on the boundary of $\mathcal{F}_{\Gamma}$
and are inside $\H$ (elliptic points) or on the boundary of $\H$
(cusps). The stabilizer of an elliptic point is cyclic of finite order
and the stabilizer of a parabolic point is infinite cyclic. The corresponding
elements of $\Gamma_{z}$ are also said to be elliptic or parabolic.
The only remaining type of elements in $\PSLR$ are the so-called
hyperbolic elements, possessing two different fixed points on $\partial\H$.
It is easy to identify the type of a M\"obius transformation from
the trace of its corresponding matrix: If $\gamma=\left(\begin{smallmatrix}a & b\\
c & d\end{smallmatrix}\right)$ (and $ad-bc=1$) then $\gamma$ is elliptic, parabolic or hyperbolic
if $\left|a+d\right|$ is less than, equal to, or greater than $2$,
respectively.

If $\mathcal{F}_{\Gamma}$ is compact, meaning that there are no cusps,
we say that $\Gamma$ is co-compact and if the hyperbolic area of
$\mathcal{F}_{\Gamma}$ is finite we say that $\Gamma$ is co-finite.
All Fuchsian groups we will encounter in this paper are co-finite
but not co-compact.

If $\Gamma'\subset\Gamma$ is a subgroup of finite index we can obtain
a fundamental domain of $\Gamma'$ by using right coset-representatives,
i.e. if $\Gamma=\sqcup_{i=1}^{\mu}\Gamma'R_{i}$ then $\mathcal{F}_{\Gamma'}=\sqcup_{i=1}^{\mu}R_{i}\mathcal{F}_{\Gamma}$
is a fundamental domain for $\Gamma'$. Here the symbol $\sqcup$
is understood to mean a disjoint union of (discrete) sets in the first
expression and a union of subsets of $\R^{2}$ with pair-wise
intersections having Lebesgue measure zero in the second expression.

\subsection{Examples of Fuchsian groups}

An elementary but important family of Fuchsian groups are the
{\em Hecke triangle groups}, $G_{n}$, generated by a reflection $S:z\mapsto-\frac{1}{z}$
and a translation $T_{\lambda}:z\mapsto z+\lambda$, where $\lambda=\lambda_{n}:=2\cos\frac{\pi}{n}$
and $n\ge3$ is an integer (for other values of $\lambda$ the resulting
group will not be Fuchsian). The standard fundamental domain, $\mathcal{F}_{n}:=\mathcal{F}_{G_{n}}$
of the Hecke triangle group $G_{n}$ can be taken as the hyperbolic
triangle with one vertex at infinity (a cusp) and the other two vertices
at the intersection of the unit circle with the vertical lines $x=\pm\frac{\lambda}{2}$
in the upper half-plane (these are elliptic points equivalent under
the group). In fact $\mathcal{F}_{n}$ also has one additional elliptic
point at $i$, meaning that $\mathcal{F}_{n}$ is in fact a rectangle,
but with one angle equal to $\pi$. 

A special case of a Hecke triangle group and the canonical example
of a Fuchsian group in number theory is the {\em modular group}, $\PSLZ$
($=G_{3}$) which has a simple presentation in terms of the generators
$S:z\mapsto-\frac{1}{z}$ and $T:z\mapsto z+1$ and relations $S^{2}=(ST)^{3}=\Id$.
The corresponding matrices are $S=\left(\begin{smallmatrix}0 & -1\\
1 & 0\end{smallmatrix}\right)$ and $T=\left(\begin{smallmatrix}1 & 1\\
0 & 1\end{smallmatrix}\right)$. 

For any integer $N\ge1$ there is a projection $\pi_{N}:\Z\rightarrow\Z/N\Z$,
$n\mapsto n\mod N$, which induces a surjective group homomorphism
$\pi_{N}:\PSLZ\rightarrow \mathrm{PSL}_{2}(\Z/N\Z)$.
Its kernel, $\Gamma(N)=\left\{ \gamma\in\PSLZ\,|\gamma\equiv\pm\Eins_{2}\mod N\right\}$,
is called the principal congruence subgroup of level $N$. A subgroup
of $\PSLZ$ which contains $\Gamma(N)$ as a subgroup,
but no $\Gamma(M)$ for $M<N$ is said to be a {\em congruence
subgroup} of level $N$. We are interested in the {\em Hecke
congruence subgroup} of level $N$, $\Gamma_{0}(N)=\left\{ \gamma=\left(\begin{smallmatrix}a & b\\
c & d\end{smallmatrix}\right)\in\PSLZ\,|\, c\equiv0\mod N\right\} $. The index of $\Gamma(N)$ and $\Gamma_{0}(N)$
in $\PSLZ$ is given by $\mu(N)=\frac{1}{2}N^{3}\prod_{p|N}\left(1-p^{-2}\right)$
and $\mu_{0}(N):=N\prod_{p|N}\left(1+p^{-1}\right)$, respectively,
where the products are taken over all primes $p|N$ (see for example \cite[Thm.~1.4.1 and (1.4.24)]{rankin:mod}). 

For a prime level, $p$, it is easy to verify directly that that the
index of $\Gamma_{0}(p)$ in $\PSLZ$ is $p+1$ and that
$R_{0}=\Id,R_{1}=S,R_{2}=ST,\ldots,R_{p}=ST^{p-1}$ is a
set of $p+1$ maps which are independent modulo $\Gamma_{0}(p)$
and therefore constitute a set of right coset representatives for
$\Gamma_{0}(p)\backslash \PSLZ$. For general non-prime
levels it is hard to find similar simple expressions. The best approach
is then to use the generators $S$ and $T$ to construct $\mu_{N}$
independent elements. This approach is also suitable for implementation on a computer.

For the purpose of this paper we need to study the subgroups $\Gamma_{0}(3)$,
$\Gamma_{0}(9)$, $\Gamma(3)$ and $\Gamma^{3}$ more closely. 
The group $\Gamma^3$ (which we will define in Example \ref{exa:Gamma^3} below) is an example of what is called a {\em cycloidal} subgroup of the modular
group, i.e. a group with only one cusp.
%(the latter will be defined below) 
\begin{example}
\label{exa:gamma09}
The index of $\Gamma_{0}(9)$ in $\PSLZ$ is $\mu(9)=12$. Hence it is clear that the set 
$\{R_{1},\ldots,R_{10}\}$ with $R_1=\Id$ and $R_{j}=ST^{j-2}$ for $j=2,\ldots,10$ 
is not sufficient as coset representatives for $\Gamma_{0}(9)\backslash\PSLZ$.
It is however easy to verify that the two maps 
$R_{11}=ST^{-3}S=\left(\begin{smallmatrix}1 & 0\\
3 & 1\end{smallmatrix}\right)$ and 
$R_{12}=ST^{3}S=\left(\begin{smallmatrix}1 & 0\\
-3 & 1\end{smallmatrix}\right)$ are independent modulo $\Gamma_0(9)$, both of each other and of the set $\{R_1,\ldots,R_{10}\}$. 

The index of $\Gamma_{0}(9)$ in
$\Gamma_{0}(3)$ is $3$ and the maps $P_{1}=\Id$,
$P_{2}=ST^{3}S$ and $P_{3}=P_{2}^{2}=ST^{6}S$ satisfy $P_{2},P_{3}\in\Gamma_{0}(3)\backslash\Gamma_{0}(9)$.
Since $P_{3}P_{2}^{-1}=P_{2}$ and $P_{2}P_{3}^{-1}=P_{2}^{-1}$ these
maps are also independent modulo $\Gamma_{0}(9)$. We have
thus shown that the following set of maps is suitable as coset representatives: 
\[
\Gamma_{0}(9)\backslash\PSLZ\simeq\left\{ R_{1},\ldots,R_{12}\right\} \quad\textrm{and}\quad\Gamma_{0}(9)\backslash\Gamma_{0}(3)\simeq\left\{ P_{1},P_{2},P_{3}\right\}.
\]
Another fact which we will use later is that $\Gamma_{0}(9)$
is conjugate to $\Gamma(3)$: If $\gamma=\left(\begin{smallmatrix}a & b\\
9c & d\end{smallmatrix}\right)\in\Gamma_{0}(9)$ then $\A_{3}\gamma \A_{3}^{-1}=\left(\begin{smallmatrix}a & 3b\\
3c & d\end{smallmatrix}\right)$ and since $ad\equiv1\mod3$ implies that $a\equiv d\equiv\pm1\mod3$
it follows that ${\A}_{3}\Gamma_{0}(9)\A_{3}^{-1}=\Gamma(3)$. 
\end{example}

\begin{example}
\label{exa:gamma3}Since the index of $\Gamma(3)$ in $\PSLZ$
is $12$ it follows that we need $3$ coset representatives for $\Gamma(3)\backslash\Gamma_{0}(3)$.
A short calculation shows that we can take $R_{1}=\Id ,R_{2}=T$ and $R_{3}=T^{2}$
as coset representatives, i.e. $\Gamma(3)\backslash\Gamma_{0}(3)\simeq\left\{\Id,T,T^{2}\right\} $. 
\end{example}

\begin{example}
\label{exa:Gamma^3}Let $h:\PSLZ\rightarrow\Z/3\Z$
be defined by $h(A)=a_{0}+a_{1}+\cdots+a_{n}\mod3$ if $A=T^{a_{0}}ST^{a_{1}}\cdots ST^{a_{n}}$.
Although this representation of $A$ is not unique it is well-known
that the only relations in $\PSLZ$ are $S^{2}=(ST)^{3}=\Id$.
Therefore the function $h$ is well-defined and in fact a surjective group homomorphism. Its kernel, 
\begin{equation}
\Gamma^{3}=\left\{ A\in\PSLZ\,|\, h(A)=0 \right\}, \label{eq:Gamma^3-def}
\end{equation}
is therefore an index $3$ subgroup of $\PSLZ$,
with coset representatives $\left\{\Id,T,T^{2}\right\} $.
By using the free generators $S$ and $R=ST$ instead of $S$ and
$T$ to define the function $h$ it is not hard to show that $S$
together with $A_{1}$ and $A_{2}$ generate $\Gamma^{3},$ where
$A_{1}=R^{-1}SR=T^{-1}ST=\left(\begin{smallmatrix}-1 & -2\\
1 & 1\end{smallmatrix}\right)$, $A_{2}=R^{-2}SR^{2}=ST^{2}ST=\left(\begin{smallmatrix}-1 & -1\\
2 & 1\end{smallmatrix}\right)$ and $A_{1}^{2}=A_{2}^{2}=\Id$. The following alternative description
is sometimes useful:
\begin{equation}
\Gamma^{3}=\left\{ A\in\PSLZ,\,\pm A\equiv\Id,S,A_{1}\,\textrm{or}\,A_{2}\mod3\right\}.\label{eq:gamma3_congdef2}
\end{equation}
Let $\tilde{\Gamma}$ denote the right hand side. Since the generators of $\Gamma^3$ is in $\tilde{\Gamma}$ 
we have $\Gamma^{3}\subseteq\tilde{\Gamma}$. It is easy to verify that any
element $A$ of $\textrm{SL}_2\left(\Z/3\Z\right)$
can be written as a product $A=BT^{m}$ where $B$ is congruent modulo
$3$ to one of the matrices 
$\left(\begin{smallmatrix}1 & 0\\
0 & 1\end{smallmatrix}\right),\left(\begin{smallmatrix}2 & 0\\
0 & 2\end{smallmatrix}\right),\left(\begin{smallmatrix}0 & 2\\
1 & 0\end{smallmatrix}\right),\left(\begin{smallmatrix}0 & 1\\
2 & 0\end{smallmatrix}\right),\left(\begin{smallmatrix}2 & 1\\
1 & 1\end{smallmatrix}\right)$,
$\left(\begin{smallmatrix}1 & 2\\ 2 & 2\end{smallmatrix}\right)$,
$\left(\begin{smallmatrix}2 & 2\\ 2 & 1\end{smallmatrix}\right)$ 
or $\left(\begin{smallmatrix}1 & 1\\
1 & 2\end{smallmatrix}\right)$ (corresponding to $\tilde{\Gamma}$) and $m=0,1$ or $2$. It follows
that $\tilde{\Gamma}$ is a congruence subgroup of level $3$ and
has index $3$ in the modular group. Since $\Gamma^{3}$ also has
index $3$ it follows that $\Gamma^{3}=\tilde{\Gamma}$. In particular we see that $\Gamma^{3}$ is indeed a congruence subgroup of level $3$. The following 
description can be deduced from \eqref{eq:gamma3_congdef2}:
\begin{equation}
\Gamma^{3}=\left\{ A=\left(\begin{smallmatrix}a & b\\
c & d\end{smallmatrix}\right)\in\PSLZ\,|\, ab+cd\equiv0\mod3\right\}. \label{eq:gamma3_congdef1}
\end{equation}
The group $\Gamma^{3}$ and other (similarly defined) groups were studied in detail by Rankin \cite[ch.~1]{rankin:mod} (see also Fricke \cite[p.~532]{48.0432.01}). 
\end{example}

\subsection{Normalizers of Fuchsian groups}
\label{sub:Normalizers-of-Fuchsian}
Let $\Gamma$ be a co-finite Fuchsian group. The {\em normalizer} of $\Gamma$
in $\PSLR$, denoted $\Norm(\Gamma)$, is the group
consisting of all elements $\gamma$ in $\PSLR$ which leaves $\Gamma$
invariant under conjugation, i.e. $\Gamma=\gamma\Gamma\gamma^{-1}$.
In certain cases it is not so hard to study the normalizer in a geometric
manner, using the concept of maximality. We say that $\Gamma$ is
maximal if there does not exist any other Fuchsian group $G$ containing
$\Gamma$ as a subgroup of finite index (this is sometimes called
finitely maximal in the literature).

From the relationship between subgroups and fundamental domains, as expressed in the last paragraph of Section  \ref{sub:hypgeom}
we see that a geometric method to find out if a given group $F$
is maximal or not is to study tessellations of a fundamental domain for $F$
into equally sized parts (with respect to hyperbolic area measure) 
identified by elements of $\PSLR$ (recall that these preserve angles
and orientation). 

The simplest example of a non-co-compact fundamental domain is the
standard domain of a of Hecke triangle group $G_{q}$, which is, as
described above, simply a triangle with one angle equal to zero and
the other two angles equal to $\frac{\pi}{q}$. Since any supergroup
of $G_{q}$ must have precisely one cusp it is immediately clear that
the only possible partition of the fundamental domain is to split
it along the imaginary axis. However, the identification of these
parts is by the reflection $J:z\mapsto-\overline{z}$ which belongs
to $\PGLR$ but not $\PSLR$. It follows that all Hecke triangle groups
are maximal. For this and related results see Greenberg \cite{MR0148620}
and Beardon \cite{beardon:isom}. 

If $\omega\in\Norm(\Gamma)$ and we let $\Gamma^{*}=\left\langle \Gamma,\omega\right\rangle$
 be the group generated by the elements of $\Gamma$ together with
$\omega$ then $\Gamma^{*}$ contains $\Gamma$ as a subgroup. It
follows that a maximal Fuchsian group must have a trivial normalizer
and to find the normalizer of a given group it is therefore enough
to find a minimal set of elements which we can use to extend $\Gamma$
to a maximal Fuchsian group. 

For $\Gamma_{0}(p)$ with $p=2$ or $3$ it is straight-forward
to check that $\Norm(\Gamma_{0}(p))=\Gamma^{*}_{0}(p)$
where 
$\Gamma^{*}_{0}(p)=\Gamma_{0}(p)\sqcup \Gamma_0(p)\omega_{p}$
with $\omega_{p}:z\mapsto-\frac{1}{pz}$ the so-called Fricke involution
(cf.~\cite[p.\ 357]{48.0432.01}): simply observe that the dilation
$\A_{\sqrt{p}}:z\mapsto\sqrt{p}z$ conjugates $\Gamma^{*}_{0}(p)$
to $G_{\sqrt{p}}$, which is maximal. It is in fact true for all primes
$p$ that 
\[
\Norm(\Gamma_{0}(p))/\Gamma_{0}(p)=\left\langle \omega_{p}\right\rangle.
\]
Note that $\omega_{p}^{2}=\Id$. Since the group $\Gamma^{*}_{0}(p)$
is in general not a triangle group the corresponding maximality argument
becomes more involved. The signature of $\Gamma^{*}_{0}(p)$
can be obtained from \cite{MR0002358} (see also \cite[pp.\ 363,366]{48.0432.01})
and we could then try to match this signature with the list of signatures
corresponding to maximal Fuchsian groups in \cite{MR0322165}. The
most general formulas for the normalizer of $\Gamma_{0}(N)$
(general $N$) are proved using a more algebraic approach with matrix
calculations. A general formula for the elements of the normalizer
is given explicitly by Lehner and Newman \cite{MR0161841} and in
a slightly more comprehensive form by Atkin and Lehner \cite{atkinlehner}.
It turns out that for each prime $Q$ dividing $N$ the normalizer
contains an involution, $W_{Q}$, of the form 
$W_{Q}=\left(\begin{smallmatrix}Qx & y\\ Nz & Qw\end{smallmatrix}\right)$ 
with $x,y,z,w\in\Z,$ $x\equiv1\mod\frac{N}{Q},$ $y\equiv1\mod Q$
and $Q^{2}xw-Nyz=Q$. These are usually called Atkin-Lehner involutions.
Although the latter papers contain correct formulas for the elements
of the normalizer, the stated formulas for the group structure of
$\Norm(\Gamma_{0}(N))/\Gamma_{0}(N)$
contains minor errors in some cases when $N$ is divisible by either $4$ or $9$.
The correct group structure can be found in e.g. Akbas and Singerman
\cite{MR1073672} or Bars \cite{MR2418382}.

For the most important group in this paper, $\Gamma_{0}(9)$,
the situation is slightly different. Consider the maps $\omega_9$ and $\St$. These are in the normalizer of $\Gamma_0(9)$ and are given by
\[
\omega_{9}:z\mapsto-\frac{1}{9z},\quad\textrm{and}\quad\St:z\mapsto z+\frac{1}{3}.
\]
Define 
$\Gamma^{*}_{0}(9)=\Gamma_{0}(9)\sqcup \Gamma_0(9)\omega_{9}\sqcup \Gamma_0(9)\St$.
We know that $\Gamma_{0}(9)$ has four inequivalent cusps,
which can be represented by $0,-\frac{1}{3},\frac{1}{3}$ and $\infty$
(cf.~e.g.~\cite[Prop.~2.6]{Iwaniec:topics}). Since $\omega_{9}(0)=\omega_{9}\St(\frac{-1}{3})=\omega_{9}\Stinv(\frac{1}{3})=\infty$
it is clear that $\Gamma^{*}_{0}(9)$ has only one cusp.
It is also readily verified that (recall that $\A_{3}z=3z$) 
$\A_{3}\omega_{9}\A_{3}^{-1}=S$, $\A_{3}\St \A_{3}^{-1}=T$ and $\A_{3}\gamma \A_{3}^{-1}\in\Gamma_{0}(3)$
for any $\gamma\in\Gamma_{0}(9)$. That is, $\A_{3}\Gamma^{*}_{0}(9)\A_{3}^{-1}=\PSLZ$,
and in particular, it follows that $\Gamma^{*}_{0}(9)$
is maximal. Since $(\St)\,^{3}=T\in\Gamma_{0}(9)$, $\omega_{9}^{2}=\Id$
and $(\omega_{9}\St)^{3}=\Id$ we see that the normalizer of $\Gamma_0(9)$ modulo $\Gamma_0(9)$ is isomorphic to 
a group with the following presentation:
\[
\Norm\left(\Gamma_{0}(9)\right)/\Gamma_{0}(9)\simeq\left\langle x,y\,|\, x^{2}=y^{3}=(xy)^{3}=1\right\rangle 
\]

From the definition \eqref{eq:Gamma^3-def} we see that
$T$ is in the normalizer of $\Gamma^{3}$. Since $S\in\Gamma^{3}$
it follows that $\Norm(\Gamma^{3})=\PSLZ$ and hence
\[
\Norm\left(\Gamma^{3}\right)/\Gamma^{3} \simeq\left\langle x\,|\, x^{3}=1\right\rangle 
\]
%since $T^{3}\in\Gamma^{3}$. 

It is also possible to consider the normalizer of $\Gamma_{0}(N)$
in $\PGLR$ instead of $\PSLR$. It is not hard to verify that $J:\mapsto-\overline{z}$
is an involution of all $\Gamma_{0}(N)$ and that in fact
the normalizer of $\Gamma_{0}(N)$ in $\PGLR$ is the group
generated by $\Norm\left(\Gamma_{0}(N)\right)$ together
with $J$.

Observe that $J$ commutes with the Fricke involutions $\omega_{N}$
but $JTJ=T^{-1}$ so the normalizer of $\Gamma^{3}$ in $\PGLR$ modulo
$\Gamma^{3}$ has the presentation $\left\langle x,y\,|\, x^{3}=y^{2}=yxyx=1\right\rangle $.

\subsection{Maass waveforms}
\label{sub:Maass-waveforms}
Let $\Gamma\subseteq\PSLZ$ be a congruence subgroup of level $N$.
It is well-known (cf.~e.g.~\cite{iwaniec:spectral}) that the discrete
spectrum of $\Delta$ on $\Gamma\backslash\H$ is spanned by
{\em Maass waveforms}, that is, real-analytic eigenfunctions of
$\Delta$ on $\H$, which are invariant under $\Gamma$. Furthermore,
in this case we know that these forms are cuspidal, i.e. vanish
at the cusps of $\Gamma$, and have finite $L^{2}$-norm induced by
the Petersson inner product:
\begin{equation}
\left\langle f,g\right\rangle _{\Gamma}=\int_{\mathcal{F}_{\Gamma}}f(z)\overline{g(z)}d\mu.\label{peterssoninnerp}
\end{equation}
The integral above can be taken over any closed fundamental domain
$\mathcal{F}_{\Gamma}$ of $\Gamma$. {\em It is tacitly understood
that, in the remainder of the paper, all mentioned Maass waveforms
are cusp forms}. The space of all Maass waveforms on $\Gamma$ with
Laplace eigenvalue $\lambda=\frac{1}{4}+R^{2}$ is denoted by $\mathcal{M}\left(\Gamma,R\right)$
and in the case $\Gamma=\Gamma_{0}(N)$ we also write $\mathcal{M}\left(N,R\right)$. 

Since the Laplacian commutes with the action of $\PGLR$ it is easy
to check that all operators (acting on Maass waveforms) which we consider
in this paper preserve the Laplace eigenvalue. \emph{We will therefore
sometimes let $\mathcal{M}\left(\Gamma\right)$ or $\mathcal{M}(N)$
denote a generic non-empty eigenspace (and omit the spectral parameter
$R$).} The theory of Hecke operators on holomorphic modular
forms carry over with only minor changes to Maass waveforms. See e.g.~\cite[III]{strombergsson:thesis}. 
\begin{defn}
\label{def:Hecke-operator-Gamma0N}For integers $n,N\ge 1$ we define
the Hecke operator $T_{n}$ of level $N$ by 
\[
T_{n}\varphi(z)=\frac{1}{\sqrt{n}}\sum_{\underset{(a,N)=1,\, a>0}{ad=n}}\sum_{b\mod d}\varphi\left(\frac{az+b}{d}\right)
\quad \textrm{for any}\quad \varphi\in\mathcal{M}(N).
\]
\end{defn}
Note that $T_1=\Id$ and with the analogous definition for $n=-1$ we get $T_{-1}=J$. 
The operators $T_{n}$ satisfy all the properties familiar from the
holomorphic theory. In particular, the set of operators $T_{n}$ with $(n,N)=1$
is a family of commuting, self-adjoint operators satisfying the following
relation for $m,n\in\Z^{+}$: 
\[
T_{m} T_{n}=\sum_{\underset{(d,N)=1,\, d>0}{d|(m,n)}}T_{\frac{mn}{d^{2}}}.
\]
Having a Hecke theory, we will now discuss the space
of oldforms and its orthogonal complement the space of newforms, in
the usual sense of Atkin-Lehner \cite{atkinlehner}. 

A fundamental lemma in the theory of oldforms is that if $M$ and 
$N$ are positive integers with $N=Md$ with $d>1$ then are are two
distinguished ways to map a function $\varphi\in\mathcal{M}(M)$ to
a function $\tilde{\varphi}\in\mathcal{M}(N)$. 
Since $\Gamma_{0}(N)\subset\Gamma_{0}(M)$ and
it follows that $\varphi\in\mathcal{M}(N)$. In addition it
is easy to verify that $\A_{d}\Gamma_{0}(N)\A_{d}^{-1}\subseteq\Gamma_{0}(M)$
and therefore $\varphi_{|\A_{d}}(z)=\varphi\left(dz\right)$ also satisfies
$\varphi_{|\A_{d}}\in\mathcal{M}(N).$ The space of old
Maass waveforms, or simply {\em oldforms}, on $\Gamma_{0}(N)$,
denoted by $\mo(N)=\oplus_{R}\mo\left(N,R\right)$, is
defined as the linear span of all forms of the type $f(z)$
and $f\left(dz\right)$ with $f\in\mathcal{M}\left(N/d,R\right)$,
$d|N$ and $d\ne1$. The {\em new space}, $\mn(N)$, is then
defined as the orthogonal complement of the old space, with respect
to the Petersson inner product. The most important property of the
decomposition into the old and new subspace is that both spaces are
invariant under the Hecke operators $T_{n}$ with $(n,N)=1$.
The following definition is then the direct analogue of the classical
definition in \cite{atkinlehner} for holomorphic forms.
\begin{defn}
\label{def:newformGamma0N}By a (normalized) \emph{Maass newform}
on $\Gamma_{0}(N)$ we mean a $\varphi\in\mn\left(N,R\right)$
which is an eigenfunction of all Hecke operators $T_{n}$ with $(n,N)=1$
and of the reflection $J$. Additionally, we assume it is normalized
to have first Fourier coefficient at infinity equal to $1$. 
\end{defn}
It is well-known that each $\mn\left(N,R\right)$ has an orthogonal
(finite) basis of Maass newforms (cf.~e.g. ~\cite[Thm.\ 4.6, p.\ 94]{strombergsson:thesis}).
By a direct calculation (or using \cite[lemma 17]{atkinlehner}) 
we see that the Hecke operators $T_{n}$ with $(n,N)=1$,
together with the Fricke involution $\omega_{N}$ and the other Atkin-Lehner
involutions $W_{Q}$, with $Q$ dividing $N$, form a commuting family
of normal, linear operators on $\mn\left(N,R\right)$. 

Since $J\left(\begin{smallmatrix}a & b\\
c & d\end{smallmatrix}\right)J^{-1}=\left(\begin{smallmatrix}a & -b\\
-c & d\end{smallmatrix}\right)$ it follows that $J$ commutes with the Hecke operators and the Fricke
involution but $JW_{Q}J^{-1}=W_{-Q}$ and $JTJ^{-1}=T^{-1}$. If we
want to desymmetrize the space $\mathcal{M}(N)$ as much
as possible we therefore might have to make a choice: either include all
Atkin-Lehner involutions, or only some of them, together with $J$.
For $\Gamma_{0}(9)$ we also have the problem that $\St$
neither commutes with $J$ nor all Hecke operators. The Hecke theory
for $\Gamma_{0}(9)$ will be studied further in Section
\ref{sub:Hecke-operators-onGamma09}. 

Analogous to the above definition for groups of the type $\Gamma_{0}(N)$
we also define the space of oldforms on $\Gamma^{3},$ denoted $\mo\left(\Gamma^{3}\right)$
as the subspace of $\mathcal{M}(\Gamma^3)$ spanned by ``lifts'' of forms from $\mathcal{M}(1)$. As we will see in Corollary
\ref{lem:The-only-lift} the only way to achieve such a form is the trivial
one (by inclusion of $\Gamma^3$ in $\PSLZ$). The space of newforms on $\Gamma^{3}$, $\mn\left(\Gamma^{3}\right)$
is then defined as the orthogonal complement of the space of oldforms.
To avoid introducing Hecke operators on $\Gamma^{3}$ we simply use
the term newform for any element of $\mn\left(\Gamma^{3}\right)$. 

In analogy with the holomorphic case we say that a Maass form $f$
is of CM-type if it possesses a self-twist, that is, if there exists
a Dirichlet character $\chi$ such that $f_{\chi}=f$. For a more
precise discussion of CM-forms with more references see e.g.~\cite[Ch.~3]{andreas:small_ev}
and \cite{CM}. The only thing we need to know about CM-forms is that
there are very few compared to the remaining forms of ``general
type'' (in fact there are none on $\Gamma_{0}(9)$).

\subsection{Twists and forms associated to other groups}
\label{sub:Lifts-and-twists}
The action of the Hecke operators (and the fact that an infinite subset of these form a commutative algebra with special multiplicative properties) is one of the most important tools
in the study of automorphic forms on congruence subgroups. It is therefore
natural to require that any systematic way of obtaining automorphic
forms on one congruence subgroup from those on another should be compatible
with this action. There are two prominent methods to do this: \emph{level raising} (i.e. producing oldforms by the inclusion 
or by the map $\A_d$) and \emph{twisting}.

Suppose that $\Gamma$ and $\Gamma'$ are two congruence subgroups, 
$\varphi\in\mathcal{M}\left(\Gamma\right)$ and 
$A\in\PGLQ$ is such that $A\Gamma'A^{-1}\subseteq\Gamma$.
Then $\varphi_{|A}\in\mathcal{M}\left(\Gamma'\right)$. It is also
easy to verify that if, for any $\alpha\in\PSLQ$, $\Gamma\alpha\Gamma$
is commensurable with $\Gamma$ (i.e. $\Gamma\alpha\Gamma\cap\Gamma$
has finite index in both $\Gamma\alpha\Gamma$ and $\Gamma$) then
$\Gamma'\tilde{\alpha}\Gamma'$ is commensurable with $\Gamma'$,
where $\tilde{\alpha}=A^{-1}\alpha A\in\PSLQ$. Thus, with the more
general definition of Hecke operators as double coset operators as
in e.g.~Miyake \cite[ch.~5]{miyake} or Shimura \cite{shimura} it
is clear that this type of correspondence is compatible with the Hecke
operators. Note that, however, the conjugated
operator $A^{-1}T_{n}A$ on $\Gamma'$ might be different from %the operator 
$T_{n}$ on $\Gamma$. 

We have already seen one example of the previous construction: In
the definition of oldforms we had $\Gamma=\Gamma_{0}(M)$,
$\Gamma'=\Gamma_{0}\left(Md\right)$ and $A=\A_{d}$. We will use the following example later.
\begin{example}
\label{exa:lift-fromGamma^3toGama09}From Example \ref{exa:gamma09}
we know that $\A_{3}\Gamma_{0}(9)\A_{3}^{-1}=\Gamma(3)\subseteq\Gamma^{3}$
and hence $\A_{3}$ provides a map from $\mathcal{M}\left(\Gamma^{3}\right)$
to $\mathcal{M}\left(\Gamma_{0}(9)\right).$ 
\end{example}
If $m$ is a positive integer, a Dirichlet character $\chi$ modulo
$m$ is a multiplicative function of period $m$ with $\chi(n)=0$
if $\left(m,n\right)>1$ and $\chi(1)=1$. The smallest
period of $\chi$ is called the conductor of $\chi$. 
%We can use Dirichlet characters to form weighted averages of automorphic
%forms, and it turns out that this process takes forms on one congruence
%subgroups to forms on another (or the same).
\begin{defn}
\label{def:twist}If $\varphi\in\mathcal{M}(N)$ and $\chi$
is a Dirichlet character of conductor $q$ we define the {\em twist
of} $\varphi$ {\em by} $\chi$ as 
\[
\varphi_{\chi}(z)=\tau\left(\overline{\chi}\right)^{-1}\,\sum_{n\mod q}\overline{\chi}(n)\varphi\left(z+\frac{n}{q}\right)
\]
where $\tau\left(\chi\right)=\sum_{n\mod q}\chi(n)e\left(\frac{n}{q}\right)$
is the usual Gauss sum (cf.~e.g.~Miyake \cite[Cor.\ 3.3.2.]{miyake}). 
\end{defn}
An alternative definition of $\varphi_{\chi}$  is that if $\varphi$
has Fourier coefficients $a_{n}$ with respect to the cusp at infinity,
then $\varphi_{\chi}$ has coefficients $\chi(n)a_{n}$
at the same cusp. It is straight-forward to verify that \cite[Lemma 4.3.10]{miyake}
also holds for Maass waveforms. From this lemma we first of all conclude
that the two definitions are equivalent. Second, we also conclude
that if $\varphi\in\mn(N)$ then $\varphi_{\chi}\in\mathcal{M}(M)$
with $M=\mbox{lcm}\left(N,q^{2}\right)$. 

The advantage of working with the first definition is that it is an
explicit formula
%completely explicit formula, valid for any automorphic form 
%(not necessarily a Hecke eigenform) 
which relates twisting by $\chi_{3}=\left(\frac{\cdot}{3}\right)$
in a natural way to the map $\St$ in the normalizer of $\Gamma_{0}(9)$.
It is otherwise more common to work with the definition in terms of
Fourier coefficients.  In this case, if $\varphi$
is a Hecke eigenform with eigenvalues $\lambda_{p}$ then $\varphi_{\chi}$
is a Hecke eigenform with eigenvalues $\chi(p)\lambda_{p}$
for $(p,N)=(p,q)=1$.

Using the definition above, an elementary matrix calculation shows directly that $\varphi_{\chi}$ is invariant under the group $\Gamma_{0}(M)$.
However, to determine the exact level of $\varphi_{\chi}$ for a general starting level seems
to be hard using only elementary methods. 
In Lemma \ref{lem:twists-are-new-and-orthog} we will see that in
the cases we need, if $\varphi$ is a newform, then $\varphi_{\chi}$ is
in fact a newform on $\Gamma_{0}(9)$.

\subsection{The Selberg trace formula for the modular group}
\label{sub:The-Selberg-Trace}

Let $\Gamma$ be a co-finite Fuchsian group and let $0=\lambda_{0}<\lambda_{1}\le\cdots$
denote the discrete spectrum of $\Gamma$ (counted with multiplicity).
We also write $\lambda_{n}=\frac{1}{4}+r_{n}^{2}$ with $r_{n}\in\R\cup i\left[0,\frac{1}{2}\right]$
and set $\sigma\left(\Gamma\right)=\{r_{n}\,|\,n\in \N_0\}$. 
Let $h\left(r\right)$ be an even analytic function and let $g(u)=\frac{1}{2\pi}\int_{-\infty}^{\infty}h\left(r\right)e^{-iur}dr$
 be the Fourier transform of $h$.
If there exists a $\delta>0$ such that $h=O(\left|r\right|^{-\left(2+\delta\right)})$
in the strip $\left|\Im(r)\right|\le\frac{1}{2}+\delta$  as $\Re(r)\rightarrow \infty$ then the
 {\em Selberg trace formula} (cf.~e.g.~\cite{hejhal:lnm548,hejhal:lnm1001}) says that 
\begin{equation}
\sum_{r_{n}\in \sigma(\Gamma)}h(r_{n})=I(\Gamma)+E(\Gamma)+H(\Gamma)+P(\Gamma)
\label{eq:STF}
\end{equation}
where $I,$ $E$, $H$, $P$ denotes the contributions of the identity,
elliptic, hyperbolic and parabolic conjugacy classes in $\Gamma$.
These terms all depend on $h$ but to simplify the notation we will omit this dependence.
By using a suitable test function it is possible to prove the analogue of Weyl's
law for the (possibly non-compact) orbifold $\Gamma\backslash\H$:
\begin{equation}
N_{\Gamma}\left(T\right)-M_{\Gamma}\left(T\right)=\frac{\mbox{vol}\left(\mathcal{F}_{\Gamma}\right)}{4\pi}T^{2}+O\left(T\ln T\right),\quad\mbox{as}\quad T\rightarrow\infty.\label{eq:weylslaw}
\end{equation}
Here $N_{\Gamma}\left(T\right)=\left|\left\{ r_{n}\in\sigma\left(\Gamma\right),\, r_{n}\le T\right\}\right| $
is the counting function for the discrete spectrum and $M_{\Gamma}(T)$
is essentially a counting function for the continuous spectrum, related
to the winding number for $\varphi_{\Gamma}$, the scattering determinant
of $\Gamma$. For a generic (non-arithmetic) Fuchsian group it is
believed (cf.~e.g.~Phillips-Sarnak \cite{phillips-sarnak:85,phillips_sarnak:85:2})
that $M_{\Gamma}$ dominates and that the discrete spectrum is at
most finite. For cycloidal and congruence
subgroups of the modular group it is possible to show that $M_{\Gamma}\left(T\right)=O\left(T\ln T\right)$
as $T\rightarrow\infty$ and therefore there exists an infinite number
of discrete eigenvalues in these cases. For an overview of Weyl's
law in different settings see e.g.~\cite{analysis-of-evolution}.
For the case of Fuchsian groups in general and congruence congruence
groups in particular see also e.g.~\cite[Thm.\ 2.28 and Ch.\ 11]{hejhal:lnm1001},
Venkov \cite{MR88h:11035} and Risager \cite{MR2098479}. 
To compute the right hand side of Weyl's law \eqref{eq:weylslaw} when $\Gamma$ is a subgroup of $\SLZ$ of index $\mu$ we recall that 
$\textrm{vol}\left(\mathcal{F}_{\SLZ}\right)=\frac{\pi}{3}$
 so that the main term is $\frac{\mu}{12}T^2$.

We will now give the explicit form of the terms above in the case
of the modular group. For simplicity we denote the corresponding terms
by $I_{1},$ $E_{1},$ $H_{1}$ and $P_{1}$ respectively. We also
write $E_{1}=E_{1}(2)+E_{1}(3)$ where $E_{1}(m)$
denotes the contribution from all elliptic classes of order $m$.
By using e.g.~\cite[p.\ 209]{hejhal:lnm1001} we obtain 
\begin{align}
I_{1}= & \frac{1}{12}\int_{-\infty}^{\infty}rh\left(r\right)\tanh\pi r\, dr,\label{eq:id-term}\\
E_{1}(2)= & \frac{1}{4}\int_{0}^{\infty}h\left(r\right)\frac{1}{\cosh\pi r}dr,\nonumber \\
E_{1}(3)= & \frac{2}{3\sqrt{3}}\int_{0}^{\infty}h\left(r\right)\frac{\cosh\frac{\pi r}{3}}{\cosh\pi r}dr\quad\mbox{and}\nonumber \\
P_{1}= & \frac{1}{4}h(0)\left[1-\varphi_{1}\left(\frac{1}{2}\right)\right]-g(0)\ln2\label{eq:parab-term}\\
 & +\frac{1}{4\pi}\int_{-\infty}^{\infty}h\left(r\right)\frac{\varphi'_{1}}{\varphi_{1}}\left(\frac{1}{2}+ir\right)dr-\frac{1}{2\pi}\int_{-\infty}^{\infty}h\left(r\right)\frac{\Gamma'\left(1+ir\right)}{\Gamma\left(1+ir\right)}dr\nonumber \end{align}
where 
\begin{equation}
\varphi_{1}(s):=\varphi_{\PSLZ}(s)=\frac{\Lambda\left(2-2s\right)}{\Lambda\left(2s\right)}\label{eq:scattering-det-for-pslz}
\end{equation}
and $\Lambda(s)=\pi^{-\frac{s}{2}}\Gamma\left(\frac{s}{2}\right)\zeta(s)$
is the completed Riemann zeta function. Although the precise form
of the hyperbolic term is not important for the proof later we will
describe it anyway, for the sake of completeness. A hyperbolic element
$P\in\PSLZ$ has two different fixed points on $\R\cup\left\{ \infty\right\} $
and acts as a dilation on the geodesic $\gamma_{P}$ which connects
them. That is, $P$ is conjugate in $\PSLR$ to a map of the form
$\A_{N\left(P\right)}:z\mapsto N\left(P\right)z$ (by virtue of the
map which takes one fixed point to $0$ and the other to $\infty$).
The number $N\left(P\right)$ is sometimes called the norm of $P$
and $\ln N\left(P\right)$ is the (hyperbolic) length of the closed
geodesic on $\PSLZ\backslash\H$ given by the projection of $\gamma_{P}$.
We use $P_{0}$ to denote the hyperbolic element of smallest norm
such that $P=P_{0}^{m}$ for some $m\in\N.$ If $P=P_{0}$
we say that $P$ is primitive. It is now clear that there is
a bijection between hyperbolic conjugacy classes and closed geodesics
on $\PSLZ\backslash\H$. The hyperbolic contribution to the Selberg
trace formula can now be written as 
\[
H_{1}=\sum_{\left\{ P\right\} }\frac{\ln N\left(P_{0}\right)}{N\left(P\right)^{\frac{1}{2}}-N\left(P\right)^{-\frac{1}{2}}}g\left(\ln N\left(P\right)\right)
\]
where the sum runs over all hyperbolic conjugacy classes in $\PSLZ$.
By collecting terms corresponding to the same length we can write
\[
H_{1}=\sum_{n=0}^{\infty}c_{n}g\left(x_{n}\right)
\]
where $0<x_{0}<x_{1}<\cdots\rightarrow\infty$ is the \emph{length
spectrum} of the modular surface, i.e. the sequence of (distinct)
lengths of its geodesics, and $c_{n}>0$ are constants. Note that
we sum over \emph{all }lengths, and not only the primitive ones. The
length spectrum of a subgroup is clearly contained in that of the
larger group. Hence, if $\Gamma$ is any finite index subgroup $\Gamma$
of the modular group then\[
H\left(\Gamma\right)=\sum_{n=0}^{\infty}c_{n}\left(\Gamma\right)g\left(x_{n}\right)\]
with the same lengths $x_{n}$ and some other non-negative constants
$c_{n}\left(\Gamma\right)$. In section \ref{sub:The-Selberg-trace-subgroups}
we will see how the other terms in the Selberg trace formula for subgroups
can be expressed in terms of those for the modular group.

\section{Maass cusp forms on $\Gamma^{3}$ and $\Gamma_{0}(9)$}

In this section we will present technical lemmas about the precise
nature of the maps between spaces of Maass waveforms introduced in Section \ref{sub:Lifts-and-twists}.
First we will discuss maps from $\PSLZ$ into $\Gamma^{3}$ and from
$\Gamma^{3}$ into $\Gamma_{0}(9)$. Then we will discuss
twisting of forms on $\PSLZ$ and $\Gamma_{0}(3)$ into
forms on $\Gamma_{0}(9)$.

\subsection{Maass waveforms on $\Gamma^{3}$}
\label{sub:Maass-waveforms-onGamma^3}
From Section \ref{sub:Normalizers-of-Fuchsian} we know that the normalizer
of $\Gamma^{3}$ has two elements $J:z\mapsto-\overline{z}$ and $T:z\mapsto z+1$,
satisfying $J^{2}=T^{3}=JTJT=\Id$ (modulo $\Gamma^{3}$). 
Using \eqref{peterssoninnerp} and a change of variables $z\mapsto z+1$ in the integral, together with the fact that the quotient $\Gamma^3\backslash \H$ 
is invariant under this map, it is easy to check that $T:f\mapsto f_{|T}$ is a unitary operator of
order three on $\mathcal{M}\left(\Gamma^{3},R\right)$. We hence obtain
an orthogonal decomposition: 
\begin{equation}
\mathcal{M}\left(\Gamma^{3},R\right)=\mathcal{M}\left(\Gamma^{3},R\right)^{(0)}\oplus\mathcal{M}\left(\Gamma^{3},R\right)^{(1)}\oplus\mathcal{M}\left(\Gamma^{3},R\right)^{(-1)}\label{eq:MGamma3-T-decomp}
\end{equation}
corresponding to the eigenvalues $1,\zeta_{3},\zeta_{3}^{-1}$ of
$T$ on this space. Here $\zeta_{3}:=e^{\frac{2\pi i}{3}}$. Since
$\PSLZ$ is generated by $S\in\Gamma^{3}$ together with $T$ 
we see immediately that 
\[
\mo\left(\Gamma^{3},R\right)=\mathcal{M}\left(\Gamma^{3},R\right)^{(0)}\quad\mbox{and\quad}\mn\left(\Gamma^{3},R\right)=\mathcal{M}\left(\Gamma^{3},R\right)^{(1)}\oplus\mathcal{M}\left(\Gamma^{3},R\right)^{(-1)}.
\]
Since $JTJ=T^{-1}$ it follows that if $\varphi\in\mathcal{M}\left(\Gamma^{3},R\right)^{(1)}$
then $\varphi_{|J}\in\mathcal{M}\left(\Gamma^{3},R\right)^{(-1)}$
(and vice versa) and we immediately deduce the following lemma. 
\begin{lemma}
\label{lem:dim-newforms-Gamma3istwo}If the space $\mn\left(\Gamma^{3},R\right)$
is non-empty, then its dimension is at least two.
\end{lemma}

\subsection{Maps into $\mathcal{M}(\Gamma^{3})$}
\label{sub:Lifts-into-.Gamma3}
An important part of the proof of the Main Theorem is the existence
of newforms on $\Gamma^{3}$ (cf. last paragraph of Section \ref{sub:Lifts-and-twists}).
The precise statement of the existence is given in Proposition \ref{prop:There-exists-R}
below and for the proof we need the next two lemmas.
\begin{lemma}
\label{lem:MinPSLZ}If $M\in\PGLQ$ is such that $MAM^{-1}\in\PSLZ$
for every $A\in\Gamma^{3}$ then $M\in\PSLZ\cup\PSLZ J$.\end{lemma}
\begin{proof}
It is well-known that any $M\in\PGLQ$ can be represented by $\left(\begin{smallmatrix}a & b\\
c & d\end{smallmatrix}\right)$ where $a,b,c,d$ are integers, $\mbox{gcd}\left(a,b,c,d\right)=1$
and the determinant is $ad-bc=r\ne0$. It turns out that it is enough
to consider the action of $M$ on the matrices $S$ and $T^{3}$ from
$\Gamma^{3}$: 
\[
MT^{3}M^{-1}=\frac{1}{r}
\left(\begin{smallmatrix}r-3ac & 3a^{2}\\
-3c^{2} & r+3ac\end{smallmatrix}\right) \quad\mbox{and}\quad MSM^{-1}=\frac{1}{r}\left(\begin{smallmatrix}ac+bd & -a^{2}-b^{2}\\
c^{2}+d^{2} & -ac-bd\end{smallmatrix}\right).
\]
Assume that $MT^{3}M^{-1}$ and $MSM^{-1}$ are both elements of $\PSLZ$
and that $p$ is a prime dividing $r$. If $p\ne3$ it is immediate
that $p$ divides $a$ and $c$ from the first matrix and that then
$p$ must divide $b$ and $d$ as well, from the second matrix. Hence
$p|\gcd(a,b,c,d)$ which is a contradiction. If $p=3$ we do not gain
any information from the second matrix but from the first we see that
$a^{2}+b^{2}\equiv c^{2}+d^{2}\equiv0\mod3$. Since $0$ and $1$
are the only squares modulo three it follows that $3$ divides $\gcd\left(a,b,c,d\right)$,
leading to a contradiction also in this case. We conclude that $\left|r\right|=1$
and if $r=1$ then $M\in\PSLZ$ and if $r=-1$ then $M\in J\PSLZ$. 
\end{proof}
Since the maps from forms on the modular group which we are interested in
are precisely those from $\PGLQ$ we immediately deduce the following
lemma. 
\begin{lemma}
\label{lem:The-only-lift}
If $\varphi \in \mathcal{M}(1)$ and $\tilde{\varphi}=\varphi_{|A}\in \mathcal{M}(\Gamma^3)$ for some $A\in \PGLQ$ then $\tilde{\varphi}=\varphi$.
\end{lemma}

\begin{prop}
\label{prop:There-exists-R}Let $N_{\Gamma^{3}}^{\n}\left(T\right)$
denote the number of Laplace eigenvalues $\lambda=\frac{1}{4}+R^{2}$, 
counted with multiplicity,
  on $\Gamma^{3}\backslash\H$ such that
the associated space of Maass waveforms consists of newforms and such
that $0\le R\le T$. Then 
\[
N_{\Gamma^{3}}^{\n}\left(T\right)\sim\frac{1}{6}T^{2}+O\left(T\ln T\right)\quad\mbox{as}\quad T\rightarrow\infty.
\]
\end{prop}
\begin{proof}
Comparing the main terms of Weyl's law \eqref{eq:weylslaw} as $T\rightarrow\infty$
for the modular group, $N_{\PSLZ}(T)\sim\frac{1}{12}T^{2}$,
with that of $\Gamma^{3}$, $N_{\Gamma^{3}}\left(T\right)\sim\frac{1}{4}T^{2}$
(see Venkov \cite{MR88h:11035}), together with Lemma \ref{lem:The-only-lift}
we see immediately that the counting function for newforms on $\Gamma^{3}$
satisfies $N_{\Gamma^{3}}^{\n}\left(T\right)=\frac{1}{6}T^{2}+O\left(T\ln T\right)$
as $T\rightarrow\infty$.
\end{proof}

\subsection{Maps from $\mathcal{M}(\Gamma^{3})$ to $\mathcal{M}(\Gamma_{0}(9))$}
\label{sub:Lifts-from-}
From the fact that $T$ acts unitarily on $\mathcal{M}(\Gamma^3)$ it follows immediately that $\St=\A_3^{-1} T \A_3$ is a unitary operator on $\mathcal{M}(9)$
and since $T$ has order $3$ on $\mathcal{M}(\Gamma^3)$ we know that $\St$ has order three on $\mathcal{M}(\Gamma_0(9))$. 
Hence we have an orthogonal decomposition of $\mathcal{M}(9)$
analogous to \eqref{eq:MGamma3-T-decomp}:
\begin{equation}
\mathcal{M}\left(9,R\right)=\mathcal{M}\left(9,R\right)^{(0)}\oplus\mathcal{M}\left(9,R\right)^{(1)}\oplus\mathcal{M}\left(9,R\right)^{(-1)}\label{eq:MGamma09-T-decomp-1}
\end{equation}
corresponding to the eigenvalue $1,\zeta_{3}$ and $\zeta_{3}^{-1}=\zeta_{3}^{2}$
of $\St$. Since $\A_{3}T\A_{3}^{-1}=\St$ it is clear that the subspace
$\mathcal{M}\left(\Gamma^{3},R\right)^{(m)}$ is mapped
into \emph{$\mathcal{M}\left(9,R\right)^{(m)}$} under
the map $\varphi\mapsto\varphi_{|\A_{3}}$. Using the notation $\mn\left(9,R\right)^{(m)}:=\mn\left(9,R\right)\cap\mathcal{M}\left(9,R\right)^{(m)}$
it follows from Proposition \ref{prop:f_in_newspace} and Lemma \ref{prop:There-exists-R}
that $\mn\left(9,R\right)^{(m)}$ (for $m=\pm1$) is non-empty
for an infinite number of values of $R$.
Consider now a newform $\varphi$ on $\Gamma^{3}$ and its associated form $\varphi_{|\A_{3}}$
on $\Gamma_{0}(9)$ (cf. Example \ref{exa:lift-fromGamma^3toGama09}).
We want to show that $\varphi_{|\A_{3}}$ is also a newform on $\Gamma_{0}(9)$
(Proposition \ref{prop:f_in_newspace}). To prove this we need to
show that $\varphi_{|\A_{3}}$ is orthogonal to all oldforms. Unfortunately,
it turns out that the space of oldforms is not invariant under $\St$, thus
we are not able to simply use the orthogonal decomposition \eqref{eq:MGamma09-T-decomp-1},
but are instead forced to compute the inner products directly. 
\begin{lemma}
\label{lem:lifts-fromGamma3toGamma09}If $\varphi\in\mathcal{M}\left(\Gamma^{3}\right)$
then $\psi=\varphi_{|\A_{3}}\in\mathcal{M}(9)$. Furthermore,
the map $\varphi\mapsto\varphi_{|\A_{3}}$ commutes with the normalizers
$T$ and $\St$.
That is, if $\varphi_{|T}=\mu\varphi$
then $\psi_{|\St}=\mu\psi.$ \end{lemma}
\begin{proof}
Since $\Gamma(3)=\A_{3} \Gamma_0(9) \A_{3}^{-1}\subset \Gamma^3$
it follows that $\psi\in \mathcal{M}(9)$ and since  
 $\A_{3}^{-1}T\A_{3}=\St$ we have  $\varphi_{|\A_{3}|\St}=\varphi_{|T|\A_{3}}$.
\end{proof}
\begin{lemma}
\label{lem:PerponGamma3andGamma09}
If $f\in\mn\left(\Gamma^{3},R\right)$ and
$g\in\mathcal{M}(3,R)$ then $f$ and  $f_{|\A_3}$  are orthogonal to $g$ 
with respect to the Petersson inner product on $\Gamma(3)$ and $\Gamma_0(9)$, respectively. 
\end{lemma}
\begin{proof}
Since $f$ is a newform we can assume, without loss of generality, that $f_{|T}=\zeta_{3}f$. 
Let $\mathcal{F}_{0}$ be a fundamental domain
for $\Gamma_{0}(3)$. By Example \ref{exa:gamma3} we see
that $\mathcal{F}=\mathcal{F}_{0}\sqcup T\mathcal{F}_{0}\sqcup T^{2}\mathcal{F}_{0}$
is a fundamental domain for $\Gamma(3)$. Since $g_{|T}=g$
we get 
\[
\left\langle f,g\right\rangle _{\Gamma(3)}  = \int_{\mathcal{F}}f\overline{g}d\mu=\int_{\mathcal{F}_{0}}f\overline{g}d\mu+
\int_{T\mathcal{F}_{0}}f\overline{g}d\mu+\int_{T^{2}\mathcal{F}_{0}}f\overline{g}d\mu
 % =
 % (1+\zeta_{3}+\zeta_{3}^{2})\int_{\mathcal{F}_{0}}f\overline{g}d\mu=0.
 = \sum_{i=0}^{2}\zeta^{i}_{3}\cdot \int_{\mathcal{F}_{0}}f\overline{g}d\mu=0.
\]
By Example \ref{exa:gamma09} we know that $\Gamma_{0}(9)\backslash\Gamma_{0}(3)\simeq\left\{ \Id,P_{2},P_{3}\right\}$
with $P_{2}=ST^{-3}S$ and $P_{3}=P_{2}^{2}=ST^{-6}S$. 
Since  $\A_{3}P_{2}^{-1}\A_{3}^{-1}=STS$
and $\A_{3}P_{3}^{-1}\A_{3}^{-1}=ST^{2}S$ it follows that $f_{|\A_{3}P_{2}^{-1}}=f_{|STS\A_{3}}=\zeta_{3}f_{|\A_{3}}$
and $f_{|\A_{3}P_{3}^{-1}}=f_{|ST^{2}S\A_{3}}=\zeta_{3}^{2}f_{|\A_{3}}$.
The same argument as above shows that  
\[
\hskip 105pt
\left\langle \smash{f_{|\A_{3}},g}\right\rangle _{\Gamma_{0}(9)} = \left(1+\zeta_{3}+\zeta_{3}^{2}\right)\int_{\mathcal{F}_{0}}f_{|\A_{3}}\overline{g}d\mu=0. 
\qed
\]
\end{proof}
\begin{prop}
\label{prop:f_in_newspace}If $f\in\mn\left(\Gamma^{3},R\right)$
then $f_{|\A_{3}}\in\mn\left(9,R\right)$. 
\end{prop}
\begin{proof}
The old space $\mo(9,R)$ is spanned by elements of  the form $g$
and $g_{|\A_{3}}$ with $g\in\mathcal{M}(3,R)$.
Since $\Gamma(3)=\A_{3}\Gamma_{0}(9)\A_{3}^{-1}$ we have 
\[
\A_{3}^{-1}\left(\Gamma_{0}(9)\backslash\H\right)=\A_{3}\Gamma_{0}(9)\A_{3}^{-1}\backslash\H=\Gamma(3)\backslash\H.
\]
Using Lemma \ref{lem:PerponGamma3andGamma09}  we now see that 
$\left\langle \smash{f_{|\A_{3}},g_{|\A_{3}}}\right\rangle _{\Gamma_{0}(9)}=\left\langle f,g\right\rangle _{\Gamma(3)}=0$
and $\left\langle f_{|\A_{3}},g\right\rangle _{\Gamma_{0}(9)}=0$.
\end{proof}
Combining Proposition \ref{prop:There-exists-R}, Proposition \ref{prop:f_in_newspace} and Weyl's law \eqref{eq:weylslaw} for the modular group and $\Gamma_0(3)$ 
a simple inclusion--exclusion argument shows that the counting function for newforms on $\Gamma_0(9)$ has the main term  
$\frac{5}{12}T^{2}$ and that the counting function for twists  
has main term $\frac{3}{12}T^2$. Hence two fifth of all newforms on $\Gamma_0(9)$ come from $\Gamma^3$.
By Lemma \ref{lem:dim-newforms-Gamma3istwo} we know that the multiplicity of the eigenspaces of Maass waveforms on $\Gamma^3$ is at least two.
Hence we can now prove one of the experimentally motivated conjectures we set out to prove, that is, the following proposition. 
\begin{mainpropa*}
Two fifth of the new part of the spectrum of $\Gamma_{0}(9)$ has multiplicity at least two.
\end{mainpropa*}
\subsection{Twists of newforms}
In the previous section we showed that newforms on $\Gamma^{3}$ are
mapped to newforms on $\Gamma_{0}(9)$. In this section
we will see that twisting of newforms from $\Gamma_{0}(3)$
and $\PSLZ$ also produce newforms of $\Gamma_{0}(9)$.
Furthermore, we will see that this contribution is orthogonal to the contribution 
from $\Gamma^{3}$. 
\begin{lemma}
\label{lem:twists-are-new-and-orthog}
If $\varphi$ is a Maass newform on $\PSLZ$ or $\Gamma_0(3)$ and $\chi=\left(\frac{\cdot}{3}\right)$
then $\varphi_{\chi}$ is a Maass newform on $\Gamma_0(9)$. Furthermore, $\varphi_{\chi}$
is orthogonal to the image of $\mathcal{M}\left(\Gamma^{3},R\right)$
under $\A_{3}$.
\end{lemma}
\begin{proof}
That $\varphi_{\chi}\in\mn\left(9,R\right)$ follows from the analogue
of \cite[Thm.\ 6]{atkinlehner} for Maass waveforms. The key point
of the proof is the multiplicity one theorem for Hecke eigenforms.
For the benefit of the reader we we include this as Theorem \ref{thm:1}
below. Consider now $f\in\mn\left(\Gamma^{3},R\right)^{(1)}$
and set $F=f_{|\A_{3}}.$ By Lemma \ref{lem:lifts-fromGamma3toGamma09}
and Proposition \ref{prop:f_in_newspace} we know that $F\in\mn(9,R)^{(1)}$,
i.e. that $F_{|\St}=\zeta_{3}F$ and $\ip{F,\varphi}_{\Gamma_0(9)}=0$. 
Using that $i\sqrt{3}\varphi_{\chi|\Stinv}=\varphi-\varphi_{|\St}$ it follows that 
\begin{align*}
\ip{F,\varphi_{\chi}}_{\Gamma_{0}(9)} & =  \zeta_{3}^{-1}\ip{F_{|\St},\varphi_{\chi}}_{\Gamma_{0}(9)}=\zeta_{3}^{-1}\ip{F,\varphi_{\chi|\Stinv}}_{\Gamma_{0}(9)}\\
 & =  \frac{1}{i\sqrt{3}}\zeta_{3}^{-1}\ip{F,\varphi}_{\Gamma_{0}(9)}-\frac{1}{i\sqrt{3}}\zeta_{3}^{-1}\ip{F,\varphi_{|\St}}_{\Gamma_{0}(9)}\\
 & =  -\frac{1}{i\sqrt{3}}\zeta_{3}^{-1}\ip{F,\varphi_{|\St}}_{\Gamma_{0}(9)}=-\frac{1}{i\sqrt{3}}\zeta_{3}^{-1}\ip{F_{|\Stinv},\varphi}_{\Gamma_{0}(9)}\\
 & =  -\frac{1}{i\sqrt{3}}\zeta_{3}^{-2}\ip{F,\varphi_{\chi}}_{\Gamma_{0}(9)}
\end{align*}
and since $-\frac{1}{i\sqrt{3}}\zeta_{3}^{-2}\ne1$ we conclude that
$\ip{F,\varphi_{\chi}}_{\Gamma_{0}(9)}=0$.  
\end{proof}
Using the results of the current section together with the standard newform theory, as introduced in Section \ref{sub:Maass-waveforms}, we can now classify all relevant maps into $\mathcal{M}(\Gamma_0(9))$.
\begin{lemma}
\label{lem:lifts-to-gamma09}
The following maps are injections into $\mathcal{M}\left(\Gamma_{0}(9),R\right)$:
\begin{alignat*}{2}
\A_{3},\A_{9},\chi_{3} &: \mathcal{M}\left(\Gamma_{0}(1),R\right)&\rightarrow& \mathcal{M}\left(\Gamma_{0}(9),R\right),\\
\A_{3},\chi_{3} &: \mn\left(\Gamma_{0}(3),R\right) &\rightarrow& \mathcal{M}\left(\Gamma_{0}(9),R\right),\\
\A_{3} &: \mn\left(\Gamma^{3},R\right) &\rightarrow& \mathcal{M}\left(\Gamma_{0}(9),R\right).
\end{alignat*}
Furthermore, the images of these maps are pair-wise orthogonal. 
\end{lemma}

\section{The main theorem}
\label{sec:The-Main-Theorem}
In section \ref{sub:Lifts-and-twists} we saw that Maass waveforms
on $\Gamma^{3}$ maps to forms on $\Gamma_{0}(9)$ in
much the same way as forms on $\Gamma_{0}(3)$ and $\Gamma_{0}(1)$
do, i.e. using maps of the form $\A_d$. Due to the standard definition of newforms, we are unfortunately
left with the newforms coming from $\Gamma^{3}$ being counted in
the space of newforms on $\Gamma_{0}(9)$. This, together
with the fact that also twists of newforms from $\Gamma_{0}(1)$
and $\Gamma_{0}(3)$ belong to the space of newforms on $\Gamma_{0}(9)$
(cf. Lemma \ref{lem:twists-are-new-and-orthog}) demonstrates that the traditional definition
of the space of newforms needs to be modified in this case. 

Consider the group $\Gamma_{0}(N)$ and recall that oldforms
on $\Gamma_{0}(N)$ are obtained by maps $\A_{d}=\left(\begin{smallmatrix}d & 0\\
0 & 1\end{smallmatrix}\right)\in\PGLQ$ with $d|N$. 
For the purpose of extending this definition we first define
the space of twists, $\mathcal{M}^{\t}(9,R)$, as the subspace of $\mn(N,R)$ spanned by elements
of the form $f_{\chi}$ with $f\in\mathcal{M}\left(d,R\right)$ with $d=1$ or $3$ and $\chi$ running through all Dirichlet characters of conductor $q$ such
that $f_{\chi}$ belongs to $\mathcal{M}\left(N,R\right)$ (cf.~Definition
\ref{def:twist}). 
We then define the space of generalized oldforms, $\mgo(9,R)$, as the space spanned by $\mathcal{M}^{\t}(9,R)$ together with 
functions of the form $f_{|A}$ where $f\in\mathcal{M}\left(\Gamma,R\right)$ for some subgroup of the modular group $\Gamma$ and 
where $A$ runs through elements of $\PGLQ$ satisfying $A\Gamma_{0}(N)A^{-1}\subseteq\Gamma$.

We now define the space of {\em genuinely new} Maass waveforms, $\mgn\left(N,R\right)$,
as the orthogonal complement of $\mgo\left(N,R\right)$ in $\mathcal{M}\left(N,R\right)$.
Since this is a subspace of the space of newforms we extend the standard
convention and say that $f$ is a genuinely new Maass cusp form if it is also
an eigenfunction of all Hecke operators as well as of the reflection
$J$. Setting $\mathcal{M}_{\Gamma^{3}}\left(9,R\right)=\A_{3}\mathcal{M}\left(\Gamma^{3},R\right)$ we can now state the Main Theorem precisely.
\begin{mainthm*}
The following is an orthogonal decomposition of the
space of Maass waveforms on $\Gamma_{0}(9)$:
\[
\mathcal{M}(9,R)=\mathcal{M}^{\o}(9,R)\oplus\mathcal{M}^{\t}(9,R)\oplus\mathcal{M}_{\Gamma^{3}}(9,R).
\]
\end{mainthm*}
\begin{proof}
By Lemma \ref{lem:lifts-to-gamma09} we can write
\[
\mathcal{M}(9,R)=\mathcal{M}^{\o}(9,R)\oplus\mathcal{M}^{\t}(9,R)\oplus\mathcal{M}_{\Gamma^{3}}(9,R)\oplus\mgn(9,R)
\]
where the last summand, $\mgn(9,R)$,  is empty by Lemma \ref{lem:gn(9R)-empty}. 
\end{proof}
From this theorem we immediately deduce that there are no genuinely new forms on $\Gamma_0(9)$, i.e. Corollary \ref{cor:n-genuinely_new_maass_forms}.
\section{The Selberg trace formula for subgroups}
\label{sub:The-Selberg-trace-subgroups}
In this section we will derive the explicit form of the terms in the
Selberg trace formula \eqref{eq:STF} for the groups $\Gamma^3$, $\Gamma_0(3)$ and $\Gamma_0(9)$. 
The reader is encouraged to review the notation of Section \ref{sub:The-Selberg-Trace}.

Let $\Gamma$ be a congruence subgroup with $v_{2}\left(\Gamma\right)$
respectively $v_{3}\left(\Gamma\right)$ elliptic conjugacy classes
of orders $2$ and $3$, $\kappa\left(\Gamma\right)$ parabolic conjugacy
classes and index $\mu(\Gamma):=\left[\PSLZ:\Gamma\right]<\infty$.
It is now straightforward to verify (cf.~e.g.~\cite[pp.\ 313-314]{hejhal:lnm1001})
that 
\begin{align*}
I\left(\Gamma\right) & =  \mu\left(\Gamma\right)I_{1},\\
E\left(\Gamma\right) & =  v_{2}\left(\Gamma\right)E_{1}(2)+v_{3}\left(\Gamma\right)E_{1}(3)\,\,\mbox{and}\\
H\left(\Gamma\right) & =  \sum_{n}c_{n}\left(\Gamma\right)g\left(x_{n}\right).
\end{align*}
It is worth repeating that the sequence $\left\{ x_{n}\right\} $
is bounded from below by $x_{0}$ -- the length of the shortest
geodesic on the modular surface, and that the $c_{n}\left(\Gamma\right)$
are non-negative. The parabolic contribution $P\left(\Gamma\right)$
is exactly the same as for the modular group \eqref{eq:parab-term},
except that $\varphi_{1}$ \eqref{eq:scattering-det-for-pslz} is
replaced by $\varphi_{\Gamma}$, the determinant of the scattering
matrix for $\Gamma$. That is: 
\begin{align*}
P\left(\Gamma\right)= & \frac{1}{4}h(0)\left[1-\varphi_{\Gamma}\left(\frac{1}{2}\right)\right]-\kappa g(0)\ln2+\frac{1}{4\pi}\int_{-\infty}^{\infty}h\left(r\right)\frac{\varphi'_{\Gamma}}{\varphi_{\Gamma}}\left(\frac{1}{2}+ir\right)dr\\
 & -\frac{\kappa}{2\pi}\int_{-\infty}^{\infty}h\left(r\right)\frac{\Gamma'\left(1+ir\right)}{\Gamma\left(1+ir\right)}dr.
\end{align*}
To relate the parabolic contribution of a subgroup to that of the
modular group we need to express $\varphi_{\Gamma}$ in terms of $\varphi_{1}$.
If $\chi$ is an even Dirichlet character with conductor $q$ then
its associated L-function, $L_{\chi}(s)$, and completed
L-function, $\Lambda_{\chi}(s)$, are given by 
\[
L_{\chi}(s)=\sum_{n=1}^{\infty}\chi(n)n^{-s},\quad\mbox{and}\quad\Lambda_{\chi}(s)=\left(\frac{\pi}{q}\right)^{-\frac{s}{2}}\Gamma\left(\frac{s}{2}\right)L_{\chi}(s).
\]
For the purpose of writing down an explicit formula for the scattering
matrix in the manner of Huxley \cite{huxley:83} we use a scaled completed
L-function 
\[
\tilde{\Lambda}_{\chi}(s)=q^{-\frac{s}{2}}\Lambda_{\chi}(s).
\]
Let $\chi=\chi_{0,q}$ denote the principal character modulo $q$,
i.e. $\chi_{0,q}(n)=1$ if $(n,q)=1$ and otherwise
$0$. From the Euler product expansion we see that $L_{\chi_{0,q}}(s)=\zeta(s)\prod_{p|q}\left(1-p^{-s}\right)$
and 
\[
\tilde{\Lambda}_{\chi_{0,q}}(s)=\pi^{-\frac{s}{2}}\Gamma\left(\frac{s}{2}\right)L_{\chi_{0,q}}(s)=\prod_{p|q}\left(1-p^{-s}\right)\Lambda(s).
\]
To simplify our later formulas we define the quotient $Q_{\chi}(s)=\frac{\tilde{\Lambda}_{\overline{\chi}}\left(2-2s\right)}{\tilde{\Lambda}_{\chi}\left(2s\right)}$
and observe that
\[
Q_{\chi_{0,q}}(s)=\prod_{p|q}\frac{1-p^{2s-2}}{1-p^{-2s}}\frac{\Lambda\left(2-2s\right)}{\Lambda\left(2s\right)}=\prod_{p|q}\frac{1-p^{2s-2}}{1-p^{-2s}}Q_{\chi_{0,1}}(s).
\]
For our purposes we need the cases $q=3$ and $9$:
\[
Q_{\chi_{0,3}}(s) = \frac{1-3^{2s-2}}{1-3^{-2s}}Q_{\chi_{0,1}}(s)\quad\mbox{and}\quad Q_{\chi_{0,9}}(s)=\frac{1-3^{2s-2}}{1-3^{-2s}}Q_{\chi_{0,1}}(s).
\]
It is now possible to use a formula for the scattering determinant of $\Gamma^{0}(N)$ (which is conjugate to $\Gamma_{0}(N)$) 
from Huxley \cite[p.\ 147]{huxley:83} to show that: 
\[
\varphi_{\Gamma_{0}(N)}(s)=A(N)^{1-2s}\prod_{\left(\chi,m\right)\in F(N)}Q_{\chi^{2}\chi_{0,m}}(s)
\]
where 
\[
A(N)=\prod_{\left(\chi,m\right)\in F}\frac{q_{\chi}N}{\left(m,N/m\right)}
\]
and the product is taken over the set 
\[
F(N)  =\left\{ (\chi,m)\,|\, m|N,\, q_{\chi}|m,\, q_{\chi}m|N\,\textrm{and}\,\chi\,\textrm{is a primitive Dirichlet character}\,\mod q_{\chi}\right\}.
\]
(The difference between our and Huxley's constant $A(N)$ arises from a difference in the normalization of the completed Dirichlet
L-functions). For $N=3$ and $9$ we have 
\begin{align*}
F(3) & =  \left\{ \left(\chi_{0,1},1\right),(\chi_{0,1},3)\right\} ,\, A(3)=9,\\
F(9) & =  \left\{ \left(\chi_{0,1},1\right),(\chi_{0,1},3),(\chi_{3},3),(\chi_{0,1},9)\right\}\,\,\mbox{and}\,\, A(9)=3^{7}.
\end{align*}
For all $\left(\chi,m\right)\in F(3)\cup F(9)$
we see that $\chi^{2}=\chi_{0,q}$ and since $q|m$ we get $\chi^{2}\chi_{0,m}=\chi_{0,m}$
for $m=1,3$ and $\chi^{2}\chi_{0,9}=\chi_{0,3}$. Using this in the formula above we get
\begin{align*}
\varphi_{\Gamma_{0}(3)}(s) & =  9^{1-2s}Q_{\chi_{0,1}}(s)Q_{\chi_{0,3}}(s)=9^{1-2s}\frac{1-3^{2s-2}}{1-3^{-2s}}\varphi_{\Gamma_{0}(1)}^{2}(s)\quad\mbox{and}\\
\varphi_{\Gamma_{0}(9)}(s) & =  \left(3^{7}\right)^{1-2s}Q_{\chi_{0,1}}(s)Q_{\chi_{0,3}}(s)^{3}=
\left(3^{7}\right)^{1-2s}\left(\frac{1-3^{2s-2}}{1-3^{-2s}}\right)^{3}\varphi_{\Gamma_{0}(1)}^{4}(s).
\end{align*}
The final component we need is the scattering determinant
for $\Gamma^{3}$. By Venkov \cite{MR88h:11035} we have
\[
\varphi_{\Gamma^{3}}(s)=3^{1-2s}\varphi_{\Gamma_{0}(1)}(s).
\]
 We can now compute the parabolic contribution for the various groups.
For simplicity we introduce the following notation: 
\begin{align*}
J_{0} & =  \frac{1}{4\pi}\int_{-\infty}^{\infty}h\left(r\right)dr \quad \textrm{and}\\
J_{3} & =  \frac{1}{4\pi}\int_{-\infty}^{\infty}h\left(r\right)\frac{d}{ds}\ln\left(\frac{1-3^{2s-2}}{1-3^{-2s}}\right)_{|s=\frac{1}{2}+ir}dr.
\end{align*}
Then, using the fact that $\varphi_{\Gamma_{0}(3)}\left(\frac{1}{2}\right)=\varphi_{\Gamma_{0}(9)}\left(\frac{1}{2}\right)=\varphi_{\Gamma^{3}}\left(\frac{1}{2}\right)=\varphi_{1}\left(\frac{1}{2}\right)=1$,
we have 
\begin{align*}
P_{1} & =  -g(0)\ln2+\frac{1}{4\pi}\int_{-\infty}^{\infty}h\left(r\right)\frac{\varphi'_{\Gamma_{0}(1)}\left(\frac{1}{2}+ir\right)}{\varphi_{\Gamma_{0}(1)}\left(\frac{1}{2}+ir\right)}dr-\frac{1}{2\pi}\int_{-\infty}^{\infty}h\left(r\right)\frac{\Gamma'\left(1+ir\right)}{\Gamma\left(1+ir\right)}dr,\\
P\left(\Gamma^{3}\right) & =  P_{1}-2\ln3\, J_{0},\\
P(3) & =  2P_{1}-4\ln3\, J_{0}+J_{3}\quad \textrm{and}\\
P(9) & =  4P_{1}-14\ln3\, J_{0}+3J_{3}.
\end{align*}

\subsection{Comparison of terms for the genuinely new trace formula }

To obtain a trace formula for the genuinely new part of the spectrum we need to compare the full contribution of $\Gamma_0(9)$ 
with the contribution of the images of the maps in Lemma \ref{lem:lifts-to-gamma09}.
Let $X(\Gamma)$, $X^{\n}(\Gamma)$ and $X^{\gn}(\Gamma)$ denote a term in the trace formula
($X=I,E,P$ or $H$) and the corresponding contribution from the newforms
and genuinely new forms. Also let $X^{*}(d)=X{}^{*}\left(\Gamma_{0}(d)\right)$.
The contribution of $\PSLZ$ to $X(9)$ is then $3$ times for the old forms and once for the twists and similarly, the new forms on $\Gamma_0(3)$ contributes $3$ times and the new forms on 
$\Gamma^3$ twice. With this notation the genuinely new term is therefore given by 
\begin{align*}
X^{\gn}(9) & =  X(9)-4X_{1}-3X^{\n}(3)-X^{\n}\left(\Gamma^{3}\right)\\
 & =  X(9)-4X_{1}-3\left[X(3)-2X_{1}\right]-\left[X\left(\Gamma^{3}\right)-X_{1}\right]\\
 & =  X(9)+3X_{1}-3X(3)-X\left(\Gamma^{3}\right).
\end{align*}
For the groups  $\Gamma^{3}$, $\Gamma_0(3)$ and $\Gamma_0(9)$ the tuple of data $(\mu,\kappa,v_{2},v_{3})$  is given by 
$(3,1,3,0)$, $(4,2,0,1)$ and $(12,4,0,0)$, respectively. Hence 
\begin{align*}
I^{\gn}(9) & =  I_{1}\left[\mu(9)+3-3\mu(3)-\mu\left(\Gamma^{3}\right)\right]=0\quad \textrm{and}\\
E^{\gn}(9) & =  0+3E_{1}(2)+3E_{1}(3)-3E_{1}(3)-3E_{1}(2)=0,\\
P^{\gn}(9) & =  4P_{1}-14\ln3\, J_{0}+3J_{3}+3P_{1}-3\left(2P_{1}-4\ln3\, J_{0}+J_{3}\right)-P_{1}+2\ln3\, J_{0}=0.
\end{align*}
Thus all terms except for possibly the hyperbolic ones cancel completely. Let
$\sigma_{\gn}\subset \sigma(\Gamma_0(9))$ denote the genuinely new part of the spectrum of $\Gamma_0(9)$. 
The trace formula \eqref{eq:STF} reduces to the following relation between the genuinely new spectrum and 
a sum over lengths of geodesics:
\begin{equation}
\sum_{r_{n}\in\sigma_{\gn}}h\left(r_{n}\right)=\sum_{n}c_{n}g\left(x_{n}\right)
\label{eq:STF-red}
\end{equation}
where $\left\{ c_{n}\right\} $ and $\left\{ x_{n}\right\} $ are
sequences in $\R$ with $x_{n}\ge x_{0}>0$ (cf.~Section
\ref{sub:The-Selberg-Trace}). To show that this part also vanishes we will make a specific choice
of test function. Let $T>0$ and consider 
\[
h_{T}(r)=\left(\frac{\sin Tr}{Tr}\right)^{4}.
\]
It is readily verified that $h_{T}$ is even and  analytic and that $h_T(r)=O(|r|^{-4})$ as $|\Re(r)|\rightarrow \infty$ in any strip of the form $|\Im(r)|<\delta$. 
Let $g_T$ be the Fourier transform of $h_T$. Then $g_T$ is easily found to be a convolution of triangle functions  and it has support
 contained in $\left[-\frac{2}{T},\frac{2}{T}\right]$.
Let $T>T_{0}=\frac{2}{x_{0}}$ and consider the trace formula \eqref{eq:STF-red} above.
Then 
\[
\sum_{r_{n}\in\sigma_{\gn}}h_{T}\left(r_{n}\right)=\sum_{n}c_{n}g_{T}\left(x_{n}\right)
\]
and the right hand side is zero since $x_{n}>\frac{2}{T}$ and thus
the right hand side is also zero. Since $h_{T}$ is non-negative and
this holds for all $T>T_{0}$ it follows that the set $\sigma_{\gn}$
is empty. We have thus shown the following lemma. 
\begin{lemma}
\label{lem:gn(9R)-empty}The space $\mgn(9,R)$
is empty for all $R$.
\end{lemma}

\section{Hecke Operators and Newforms on $\Gamma_{0}(9)$ }
\label{sub:Hecke-operators-onGamma09}
We now want to study how the decomposition \eqref{eq:MGamma09-T-decomp-1},
into eigenspaces of the operator $\St$, behaves under the action
of the Hecke operators $T_{p}$ defined in Section  \ref{sub:Maass-waveforms}.
It turns out that the mapping properties of $T_{p}$ depend on $p$
modulo $3$ and we have the following lemma.
\begin{lemma}
\label{lem:property-of-Tp}The family of Hecke operators
$\left\{ T_{p}\right\} _{p\ne3}$ on $\Gamma_{0}(9)$ has
the following properties:
\begin{enumerate}[(a)]
\item $JT_{p}J=T_{p}$ for all primes $p\ne3$, 
\item $T_{p}\St=\St T_{p}$ if $p\equiv1\mod3$ and 
\item $T_{p}\St=\Stinv T_{p}$ if $p\equiv2\mod3$. 
\end{enumerate}
\end{lemma}
\begin{proof}
Using Definition \ref{def:Hecke-operator-Gamma0N} we can write the action of $T_{p}$ on $f\in\mathcal{M}(9)$
as 
\[
T_{p}f=\sum_{b\mod p}f_{|\beta_{p,b}}+f_{|\alpha_{p}}
\]
 where $\alpha_{p}=\left(\begin{smallmatrix}p & 0\\
0 & 1\end{smallmatrix}\right)$ and $\beta_{p,b}=\left(\begin{smallmatrix}1 & b\\
0 & p\end{smallmatrix}\right)$. 
A direct computation shows that $J\beta_{p,j}J=\beta_{p,-j}$ and $J\alpha_{p}J=\alpha_{p}$.
It follows that $T_{p}$ commutes with $J$ for all primes $p$. 
By representing the operator $\St$ by the matrix $\left(\begin{smallmatrix}1 & \frac{1}{3}\\
0 & 1\end{smallmatrix}\right)$ we check that if $p\equiv1\mod3$ then $\beta_{p,b}T^{\frac{1}{3}}=T^{\frac{1}{3}}\beta_{p,b'}$
with $b'\equiv b+\frac{1-p}{3}\mod p$ and $\alpha_{p}T^{\frac{1}{3}}=T^{\frac{p-1}{3}}T^{\frac{1}{3}}\alpha_{p}$.
If $p\equiv2\mod3$ then $\beta_{p,b}T^{\frac{1}{3}}=T^{\frac{2}{3}}\beta_{p,b'}$
with $b'\equiv b+\frac{2-p}{3}\mod p$ and $\alpha_{p}T^{\frac{1}{3}}=T^{\frac{p-2}{3}}T^{\frac{2}{3}}\alpha_{p}$. 
\end{proof}
From this lemma we conclude that if $p\equiv1\mod3$ then 
\[
T_{p}:\mn\left(9,R\right)^{(m)}\rightarrow\mn\left(9,R\right)^{(m)}\quad\textrm{for}\quad m=0,1,-1
\]
and if $p\equiv2\mod3$ then 
\begin{alignat*}{2}
T_{p}&:\mn\left(9,R\right)^{(0)}  &\rightarrow &\mn(9,R)^{(0)}\quad\mbox{and}\\
T_{p}&:\mn\left(9,R\right)^{(\pm1)}&  \rightarrow &\mn(9,R)^{(\mp1)}.
\end{alignat*}
\begin{rem}
We disregard the prime $p=3$ in this discussion because
$T_{3}$ is identically zero on $\mn\left(\Gamma_{0}(9),R\right)^{\left(\pm1\right)}$.
To verify this, observe that $T^{\frac{1}{3}}\beta_{3,j}=\beta_{3,j+1}$
and hence $T_{3}\St=T_{3}$. Therefore, either
$\St$ acts trivially or $T_{3}$ acts as the zero operator.
\end{rem}
Since the decomposition into eigenspaces of $\St$ is a key step in
proving multiplicity we consider a modified family of Hecke operators
which preserves this decomposition: 
\[
\mathcal{T}_{p}=T_{p}\quad\textrm{if}\quad p\equiv1\mod3 \quad\textrm{and}\quad \mathcal{T}_{p}=JT_{p}\quad\textrm{if}\quad p\equiv2\mod3.
\]
\begin{lemma}
The family of operators $\mathcal{T}=\left\{ \mathcal{T}_{p}\right\} _{p\ne3}$
has the following properties:
\begin{enumerate}
\item All $\mathcal{T}_{p}$ are normal and pair-wise commuting. 
\item All $\mathcal{T}_{p}$ preserves $\mn\left(9,R\right)^{(m)}$
for $m=0,1,-1$. 
\end{enumerate}
\end{lemma}
\begin{proof}
The first property follows from the corresponding property of the
Hecke operators $T_{p}$ (cf.~e.g.~\cite[ch.\ 4.5]{miyake}) together
with the observations that$J^{2}=\Id$ and $p^{2}\equiv1\mod3$ for
any prime $p\ne3$. The second property follows from Lemma \ref{lem:property-of-Tp}
combined with the fact that $J$ intertwines the spaces with $m=1$
and $-1$. 
\end{proof}
Since $\mathcal{T}$ consists of commuting normal (recall that an operator on a Hilbert space is said to be normal if it commutes with
its adjoint)  operators it follows that $\mn(9,R)^{(1)}$
has an orthonormal basis $\left\{\Phi_{1},\ldots,\Phi_{h}\right\} $
 of simultaneous eigenfunctions of all operators in 
$\mathcal{T}$. If $\Psi_{k}=\Phi_{k|J}$ then $\left\{ \Psi_{1},\ldots,\Psi_{h}\right\} $ is also an orthonormal
$\mathcal{T}$-eigenbasis of $\mn\left(9,R\right)^{(-1)}$.
We now define 
\[
F_{j}^{+}=\frac{1}{2}\left(\Phi_{j}+\Psi_{j}\right),\quad\mbox{and}\quad F_{j}^{-}=\frac{1}{2}\left(\Phi_{j}-\Psi_{j}\right),\,1\le j\le h.
\]
It is easy to verify that $\left\{ F_{j}^{\pm}\right\} _{1\le j\le h}$
is an orthonormal $\mathcal{T}$-eigenbasis of 
$$\bigoplus_{m=\pm1}\mn\left(9,R\right)^{(m)}$$
and that $JF_{j}^{\pm}=\pm F_{j}^{\pm}$. Hence $F_{j}^{+}$ and
$F_{j}^{-}$ are eigenforms of the Hecke operators
$T_{p}$, $p\ne3$ and hence also of all $T_{n}$ with $(n,3)=1$.
Suppose that $\Phi_{j}$ has a Fourier expansion \[
\Phi_{j}(z)=\sum_{n\ne0}a_{j}(n)\kappa_{n}(y)e(nx)\]
where $\kappa_{n}\left(y\right)=\sqrt{y}K_{iR}\left(2\pi\left|n\right|y\right)$
and $e(x)=e^{2\pi ix}$. Then 
\begin{align*}
\Phi_{j|\smash{T^{1/3}}}(z) & = \sum_{n\ne0}a_{j}(n)\kappa_{n}(y)e\left(n\left(x+\smash{\scriptstyle \frac{1}{3}}\right)\right)=
\sum_{n\ne0}a_{j}(n)\zeta_{3}^{n}\kappa_{n}(y)e(nx)\quad\mbox{and}\\
\Psi_{j}(z) & =\Phi_{j}(-x+iy) = \sum_{n\ne0}a_{j}(n)\kappa_{n}(y)e(-nx)=\sum_{n\ne0}a_{j}(-n)\kappa_{n}(y)e(nx).
\end{align*}
Since $\Phi_{j|\smash{T^{1/3}}}=\zeta_{3}\Phi$ it follows that $a_{j}(n)=0$
unless $n\equiv1\mod3$ and 
\begin{align*}
F_{j}^{+}(z) & = \sum_{(n,3)=1}c_{j}^{+}(n)\kappa_{n}(y)e(nx),\quad c_{j}^{+}(n)=
\begin{cases}
a_{j}(n), & n\equiv1\mod3,\\
a_{j}(-n), & n\equiv2\mod3,
\end{cases} \\
F_{j}^{-}(z) & = \sum_{(n,3)=1}c_{j}^{-}(n)\kappa_{n}(y)e(nx),\quad c_{j}^{-}(n)=
\begin{cases}
a_{j}(n), & n\equiv1\mod3,\\
-a_{j}(-n), & n\equiv2\mod3.
\end{cases}
\end{align*}
In other words, we have $c_{j}^{-}(n)=\left(\frac{n}{3}\right)c_{j}^{+}(n)$
for all $n$. The functions $F_{j}^{\pm}$ are not identically zero
since $\Phi_{j}$ is not identically zero. 
Furthermore,  $F_j^{+}$ is orthogonal to $F_j^{-}$ since $J$ is an involution.
 Proposition \ref{prop:Hecke-pair} of the introduction (repeated below) is now an immediate consequence
of the construction of the basis $\left\{ \smash{F_{j}^{\pm}}\right\} $
of $\bigoplus_{m=\pm1}\mn\left(9,R\right)^{(m)}$ together
with Weyl's law for newforms on $\Gamma^{3}$, cf.~Lemma \ref{prop:There-exists-R}. 

\begin{mainprop*}
There exist an infinite number of pairs $\left\{ F^{+},F^{-}\right\}$
of Maass newforms on $\Gamma_{0}(9)$ with the property
that the Hecke eigenvalues $c^{+}(n)$ and $c^{-}(n)$ of $F^{+}$ and $F^{-}$ are related through 
\[
c^{-}(n)=\left(\frac{n}{3}\right)c^{+}(n)
\]
for all $n$ relatively prime to $3$. 
\end{mainprop*}

\section{Some perspectives from representation theory}
\label{sec:Some-Perspectives-from}
\subsection{A representation-theoretical interpretation of the main theorem}
\label{sec:A-representation-theoretic-inter}
It is a well-known fact that a modular form $f$ of weight $k$ on a subgroup $\Gamma \subseteq \PSLZ$ corresponds to a vector-valued modular form $F$ on $\PSLZ$, 
transforming with respect to a certain finite dimensional representation $\rho_{F}$.
To be more precise, if $\Gamma\backslash\PSLZ=\left\{V_i\right\}$ and $F$ is taken as the vector with components $F_i=f_{|_{k}V_i}$ then $F$ will transform according to 
the induced representation of $\Gamma$. The matrix coefficients of this can be defined by $\rho_F(A)_{ij}=1$ if $V_{i}AV^{-1}_{j}\in\Gamma$ and $0$ otherwise. 
The representation $\rho_{F}$ is in general not irreducible. 
One way of obtaining information about the original modular form $f$ is to study the decomposition of $\rho_F$ into irreducible components. 
%is one way of obtaining information about the original modular form $f$.  
It turns out that in the present case, that of Maass waveforms on $\Gamma_0(9)$, the representation $\rho_F$ can be viewed as a representation of the symmetric group $\S_4$. 
The purpose of this section is to demonstrate how the irreducible representations of $\S_4$ are related to the %associated to the 
different types of functions appearing in the decomposition of the space $\mathcal{M}(9)$ given in Theorem \ref{thm:main_thm}. 

For technical reasons it is easier to work with the general linear group rather than the special linear group.
If $\Gamma\subseteq\PSLZ$ then we let $\overline{\Gamma}\subseteq\PGLZ$ denote the 
the group generated by the elements of $\Gamma$ together with the reflection $J$. There are two canonical characters 
on $\PGLZ$, namely the trivial character, $\chi_{0}$, and the sign of
the determinant, $\chi_{\sgn}=\sgn \circ \det $.

Fix a non-zero $f\in\mathcal{M}(9)$
which transforms under the character $\chi_{0}$ or $\chi_{\sgn}$
on $\overline{\Gamma_{0}(9)}$, that is, with the notation used in the introduction, $f$ is either
\emph{even} or \emph{odd}. Since $\A_{3}\Gamma_{0}(9)\A_{3}^{-1}=\Gamma(3)$
we see that $f_{|\A_{3}^{-1}}\in\mathcal{M}(\Gamma(3))$.
Set $v_{\gamma}=f_{|\A_{3}^{-1}\gamma}$ and consider the complex vector space
\[
\VV\left(f\right)=\left\{ v_{\gamma}\,|\,\gamma\in\Gamma(3)\backslash\PGLZ\right\}
\]
equipped with the following $\PGLZ$-action: $A.v_{\gamma}:=v_{\gamma A^{-1}}$.
Since $\Gamma(3)$ is a normal subgroup it is clear that the elements $\gamma$ can be taken as simultaneous right- and left-coset representatives and hence  
 $B.v_{\gamma}=v_{\gamma B^{-1}}=v_{\gamma}$ if $B\in \Gamma(3)$. 
It follows that $\VV(f)$ is a complex, non-zero, finite-dimensional representation of $\PGLZ/\Gamma(3)$.

Let  $\mathbb{F}_{3}$ denote the finite field with $3$ elements and consider the group homomorphism $\varphi_3:\PGLZ \rightarrow \PGLF{3}$
given by reducing each matrix entry mod $3$. 
 Since $\varphi_3$ is surjective and its kernel is $\Gamma(3)$ it follows immediately that $\PGLZ/\Gamma(3)\simeq\PGLF{3}$.
It is easy to see that the maps $\bar{S}:=\varphi_3(S)$, $\bar{T}:=\varphi_3(T)$ and $\bar{J}:=\varphi_3(J)$ are 
generators of $\PGLF{3}$.
A short calculation shows that the map $h:\PGLF{3}\rightarrow \S_{4}$ defined on the generators by $h(\bar{S})=(1 2)(3 4)$, $h(\bar{T})=(1 2 3)$ and $h(\bar{J})=(1 2)$
gives an isomorphism between $\PGLF{3}$ and the symmetric group $\S_4$.
Due to this isomorphism we are able to obtain all necessary information about the irreducible representations of $\PGLF{3}$ from standard references on representations of finite groups, for example Fulton and Harris \cite{MR1153249}. 

It is clear that $\VV(f)$ can be viewed as a representation of $\S_{4}$, and
we know that $\S_{4}$ has five irreducible representations: $\chi_{0}$, $\chi_{\sgn}$, $WS$, 
$\rho_{\std}$, and $\rho_{\std}\otimes\chi_{\sgn}$, of dimensions
$1$, $1$, $2$, $3$ and $3$, respectively.
It remains to find out the exact correspondence between these irreducible representations 
and properties of the function $f$. 
If $f$ is an oldform related to a specific newform $g$ of lower level then we view the representation 
generated by \emph{all} oldforms obtained from $g$ as the natural object associated to $f$, instead of just the representation $\VV(f)$ defined above. 
To be precise, if $g$ has level $1$ and $f\in \textrm{span}\left\{g,g_{|\A_3},g_{|\A_9}\right\}$ then we set 
$\VV(f):=\VV(g,g_{|\A_3},g_{|\A_9})=\left\{ {v_i}_{\gamma}\,|\,\gamma\in\Gamma(3)\backslash\PGLZ,\, i=0,1,2\right\}$,
where $v_i=g_{|{\A}^{i-1}_3}$, and analogously if $g$ has level $3$. 
By identifying $S_4$ and $\PGLF{3}$ as above, and using the character table of $\S_{4}$, it is not hard to show the following properties explicitly:
\begin{enumerate}
\item \label{id:old1}If $f$ is an oldform associated to an even newform of level $1$ then $\VV(f)=\chi_{0} \oplus\rho_{\std}$.
\item \label{id:old3}If $f$ is an oldform associated to an even newform of level $3$ then $\VV(f)=\rho_{\std}$.
\item \label{id:twist}If $f$ is the twist of an even newform of level $1$ or $3$ then $\VV(f)=\rho_{\std}$.
\item \label{id:W}If $f$ is the lift of a newform on $\Gamma^{3}$ then $\VV(f)=W$. 
\end{enumerate}
In the first three cases, an odd instead of an even newform corresponds to a twist of the representation by $\chi_{\sgn}$. 
To prove \ref{id:W} it is helpful to use the description of $\Gamma^3$ in terms of congruences  \eqref{eq:gamma3_congdef2}
to show that the kernel of $W$ is precisely $\Gamma^3$ (viewed as a subgroup of $\PGLF{3}$). 
From \ref{id:old1}-\ref{id:W} we see that there is a relationship between the  irreducible representations of $\S_4$ and the 
different constituents in the decomposition of $\mathcal{M}(9)$ given by Theorem \ref{thm:main_thm}.
However, to \emph{prove} that such a decomposition holds we would need a one-to-one correspondence, something which seems   
to be out of reach using the ``classical`` approach outlined above. 

To use representation theory to prove that Theorem \ref{thm:main_thm} holds, or equivalently, that there are no genuinely new forms on $\Gamma_{0}(9)$,
we have to consider automorphic representations of $\GL{A}$ where $\AA$ is the ring of adeles over $\Q$. 
A precise formulation of
all definitions and results in this area would be too lengthy and
take us far out of the scope of this paper. We therefore simply 
outline the essential parts of the argument and leave the technicalities
to the interested reader. The necessary background can be obtained from, for example, Gelbart
\cite{gelbart:autom_adele}. Let $\mathbb{Z}_{p}$ be the ring of
integers in $\mathbb{Q}_{p}$, the $p$-adic completion of $\mathbb{Q}$
at the finite prime $p$. For a positive integer $N$ we choose subgroups $K_{0}(N)$,
$K_{0}^{0}(N)$, $K(N)$ and $K^{3}$ of $\GLZhat=\prod_{p<\infty}\GLZp$
satisfying $\GLRplus K\cap\mbox{GL}_{2}(\mathbb{Q})=\Gamma_{0}(N)$,
$\Gamma_{0}^{0}(N)$, $\Gamma(N)$ and $\Gamma^{3}$, respectively. 
The choice is made such that the determinant map is surjective to $\Z_{p}^{\times}$, and hence such that  
strong approximation holds, that is, $\GL{A}=\GL{Q} K \GLRplus$ when $K$ is any of the above groups. %$K$ being any of these groups.

Let $\pi$ be an  automorphic representation and write $\pi^{K}$ for
the set of vectors of $\pi$ which are fixed under some $K\subseteq\GLZhat$.
It is well-known that there is a unique newform attached to each irreducible component of $\pi$. % and if this newform has level $N$ we say that $\pi$ has level $N$. 
If $\pi^{K_{0}(9)}\ne\left\{ 0\right\} $ then we use
$\A_{3}\in\GLRplus$ to show that $\pi^{K(3)}\ne\left\{ 0\right\}$ and it is easy to see that $\VV(\pi):=\pi^{K(3)}$
can be viewed as a complex, non-zero, finite-dimensional representation of $\PGLF{3}$.
We want to show that none of the possible irreducible constituents of $\VV(\pi)$ corresponds to a genuine newform. 

It clear that if $\VV(\pi)$ contains the trivial
representation then it contains a vector fixed under $\GLZhat$
and the associated newform has level $1$. 
If $\VV(\pi)$ contains $\rho_{\std}$ then it is easy to check that it contains a vector which fixed under $T$ and therefore under the 
 whole group $K_{0}(3)$. Hence the associated newform has level $3$.
If $\VV(\pi)$ contains $\chi_{\sgn}$ or $\chi_{\sgn}\otimes\rho_{\std}$
then $\VV(\pi)\otimes \chi_{\sgn}=\VV(\pi\otimes\chi_{3})$ contains $\chi_{0}$ or $\rho_{\std}$,
and hence $\pi$ is a twist of an automorphic representation associated to a newform of level
$1$ or $3$. For the last statement, observe that for $A\in\GLF{3}$ we have $\chi_{3}(\det(A))=\sgn(\det(A))$,
 and by definition $\pi\otimes\chi_{3}$
consists of functions of the form $g\mapsto\varphi(g)\chi_{3}(\det(g))$
with $\varphi\in\pi$. Finally, if $\VV(\pi)$ contains $W$ then
$\pi$ contains a non-zero vector fixed under the subgroup $K^{3}\subset\GLZhat$
and the newform attached to $\pi$ is therefore a newform on $\Gamma^3$. 

Since one of these five possibilities must occur we conclude that any newform attached to $\pi$ is in fact a newform on some subgroup of lower level, 
The corresponding ``classical'' statement for a Maass form, or modular form, $f$, follows
by considering $\pi_{f}$, the representation generated by all the $\GL{A}$-translates
of the automorphic form $\varphi_{f}$ associated to $f$. See for
example \cite[section 5.C]{gelbart:autom_adele}. 

The connection with representation theory of ${\S}_{4}$ was first mentioned to the author by A. Mellit.
A more precise description of the idea behind this correspondence, together with details regarding the automorphic point of view 
was given to the author by K. Buzzard (private communication).

\subsection{Multiplicity one}
A fundamental property in the theory of holomorphic newforms is that they are uniquely determined by their Hecke eigenvalues, $\lambda_n$. 
In fact, Atkin and Lehner \cite[Theorem 4]{atkinlehner} showed that they are uniquely determined by any infinite subset of the Hecke eigenvalues with prime index. 
This is the first example of a \emph{multiplicity one} theorem. There has been  many generalizations of this theorem, both formulated classically and 
in in terms of automorphic representations. 
Using  Gelbart \cite[Thm.\ 5.19]{gelbart:autom_adele} to translate results from the language of automorphic representations to that of Maass waveforms 
we will now describe some of the strongest multiplicity one theorems which are known at the moment.
\begin{thm}
[Jacquet-Shalika, Gelbart, Ramakrishnan and Rajan]\label{thm:1}Let
$\varphi\in\mn(N,R)$ and $\psi\in\mn(N',R)$
be two Maass newforms with identical Laplace eigenvalue but possibly
different levels $N$ and $N'$. Let $\lambda_{p}$ and $\lambda_{p}'$
denote their respective $T_{p}$-eigenvalues and define $S$ to be
the set of primes where these eigenvalues are different: \[
S=\left\{ p\quad\mbox{prime}\,|\,\lambda_{p}\ne\lambda'_{p}\right\} .\]
Let $D\left(S\right)$ be the Dirichlet (or analytic) density of $S$,
defined by \[
D\left(S\right)=\lim_{s\rightarrow1^{+}}\frac{1}{\ln\left(\frac{1}{s-1}\right)}\sum_{p\in S}p^{-s}.\]
Then 
\begin{enumerate}[(a)]
\item  If $S$ is finite then $N=N'$ and $f=g$.\label{mpt-a}
\item  If $D\left(S\right)<\frac{1}{8}$ then $N=N'$ and $f=g$. \label{mpt-b}
\item  If $\sum_{p\in S}p^{-\frac{2}{5}}<\infty$ then $N=N'$ and\label{mpt-c}
$f=g$. 
\end{enumerate}

\end{thm}
The ,,standard`` strong multiplicity one theorem (\ref{mpt-a}) follows from
e.g. Jacquet-Shalika \cite[Thm.\ 4.8]{MR618323} or Gelbart \cite[Thm.\ 5.12]{gelbart:autom_adele}.
The refinements obtained in (\ref{mpt-b}) and (\ref{mpt-c}) were shown by Ramakrishnan
\cite{MR1253208} and Rajan \cite[pp.\ 188-189]{MR1994478} respectively.
See also \cite{MR1707005}. A slightly different flavor of multiplicity
one is given by the following theorem of Ramakrishnan \cite[Cor.\ 4.1.3]{MR1792292}.
All notations are as in Theorem \ref{thm:1}.
\begin{rem}
The above theorem is also true for forms with Nebentypus (Dirichlet
characters). In this case, if $\varphi$ and $\psi$ are assumed to
have Nebentypus $\chi$ and $\chi'$ then the conclusions (\ref{mpt-a})-(\ref{mpt-c})
also include the statement that $\chi=\chi'$. \end{rem}
\begin{thm}
[Ramakrishnan]\label{thm:2}If $\lambda_{p}^{2}=\lambda_{p}'^{2}$
for every prime $p,$ $\left(NN',p\right)=1$ then there exists a
Dirichlet character $\chi$ such that $\lambda_{p}'=\chi(p)\lambda_{p}.$
\end{thm}
From the corresponding results in the $l$-adic context \cite{MR1606395}
it is conjectured (cf.~e.g.~\cite[p.\ 189]{MR1994478}) that the
 ``critical density'' $\frac{1}{8}$ in Theorem \ref{thm:1} (\ref{mpt-b})
above can be replaced with $\frac{1}{2}$ if $f$ and $g$ are not
of CM-type. 

With notation as in the previous theorems assume that $\lambda_{p}'=\chi(p)\lambda_{p}$
with a non-trivial Dirichlet character $\chi$ mod $M$. Then $S=S_{\chi}$
where 
\begin{align*}
S_{\chi} & = \left\{ p\quad\mbox{prime}\,|\,\chi(p)\ne1\right\} \quad\mbox{and}\\
D\left(S_{\chi}\right) & =  \frac{1}{\phi(M)}\left|\left\{ p\in\left(\Z/M\Z\right)^{*}\,|\,\chi(p)\ne1\right\}\right| 
\end{align*}
by Serre \cite[Thm.\ 2, Ch.\ VI]{MR0344216}. Since $\sum_{d\mod M}\chi\left(d\right)=0$ unless $\chi$ is trivial we see that the set
$\left\{ \chi\left(d\right)\,|\, d\in\left(\Z/M\Z\right)^{*}\right\} $
contains at most $\frac{\phi(M)}{2}$ ones and thus $D\left(S_{\chi}\right)\ge\frac{1}{2}$. 
It is therefore not expected
that the inequality, $D\left(S\right)<\frac{1}{8}$, of Theorem \ref{thm:1}
can be replaced by any inequality stronger than $D\left(S\right)<\frac{1}{2}$
even for generic forms. 

Now, let $\left\{ F^{+},F^{-}\right\} $ be a pair of newforms on
$\Gamma_{0}(9)$ as in Proposition \ref{prop:Hecke-pair},
i.e. having Hecke eigenvalues $\lambda_{p}^{-}=\chi_{3}(p)\lambda_{p}^{+}$
with $\chi_{3}(p)=\left(\frac{p}{3}\right)$. Let $\pi^{+}$
and $\pi^{-}$ be the automorphic representations corresponding to
$F^{+}$ and $F^{-}$ and let $\Pi^{+}$ and $\Pi^{-}$ be the base-change
of $\pi^{+}$ and $\pi^{-}$ with respect to the quadratic extension
$\Q\left(\zeta_{3}\right)$. The associated character of this
extension is $\chi_{3}$ and since $\pi^{-}=\chi_{3}\otimes\pi^{+}$
it follows that $\Pi^{+}=\pi^{+}\oplus\chi_{3}\otimes\pi^{+}$ and
$\Pi^{-}=\pi^{-}\oplus\chi_{3}\pi^{-}$ are equal. Or, in other words,
the base-changed forms $E^{+}$and $E^{-}$ of $F^{+}$and $F^{-}$
have identical L-functions.

\section{Concluding remarks and possible extensions}

A natural question to ask at this point is whether the results of
this paper, in particular the non-existence of genuinely new forms
and the existence of multiple ``new`` eigenvalues, are specific to $\Gamma_{0}(9)$
or if similar arguments are applicable to other (square) levels. 

An essential ingredient in the proof of the existence of multiple eigenvalues in the new spectrum on $\Gamma_0(9)$ was the normalizer $\St$.
In particular, that we could prove the existence of eigenfunctions of this operator, with eigenvalue different from $1$.
According to Atkin and Lehner \cite[Lemma 29]{atkinlehner} the normalizer of
$\Gamma_{0}(N)$ contains maps of the form $T^{\frac{1}{q}}$
if $q^{2}||N$ and $q=2,3,4,8$. Recall the argument which forces
a higher multiplicity: if $f$ is an eigenfunction of $\St$ with
eigenvalue $\mu\ne1$ then $f_{|J}$ is also an eigenfunction of $\St$,
but with eigenvalue $\mu^{-1}\ne\mu$. To prove multiplicity in the case of even
$q$ we must prove the existence of eigenfunctions with eigenvalues not equal to $1$
or $-1$.

For the non-existence of genuinely new Maass forms we have displayed two different approaches, one ``classical``, using orthogonality relations and the Selberg trace formula, 
and one ''automorphic``, using representation theory.  
For both of these approaches an important step was that we found the group $\Gamma^3$ between $\PSLZ$ and 
 $\Gamma_0^0(3)$, with the latter group conjugate to $\Gamma_0(9)$. 
Because of the simple form of $\Gamma^3$ we could prove that newforms on $\Gamma^3$ lifts to newforms on $\Gamma_0(9)$
and we could also express all relevant terms in the Selberg trace formula explicitly. 

To use the classical approach in this paper to study a the group $\Gamma_0(N^2)$ for an arbitrary integer $N$ a necessary first step is to determine all its conjugates in $\PSLZ$, as well as their 
supergroups and possible lifting maps. 
Once this is done, a hint to whether there are genuinely new forms or not can be obtained from the indices of the groups involved, together with careful book-keeping.
A complete proof clearly requires more effort, and if the intermediate groups are not cycloidal, or otherwise admit simple descriptions, it might be 
 hard to prove all required lemmas. % and evaluating the relevant terms in the Selberg trace formula.
Unfortunately we know by a result of Petersson \cite{MR0294255} that there is in fact only a finite number of cycloidal congruence groups, all with indices dividing
 $55440=2^{4}\cdot3^{2}\cdot5\cdot7\cdot11$. For general $N$ it is therefore always necessary to investigate more complicated groups.

In the representation-theoretical approach described in Section \ref{sec:A-representation-theoretic-inter} the key components are the isomorphisms 
$\Gamma^0_0(3)\backslash \PGLZ\simeq \PGLF{3}$ and $\PGLF{3}\simeq \S_{4}$. 
From these we immediately got an explicit description of all irreducible representations of $\Gamma_0^0(3)\backslash \PGLZ$. 
For a prime $p>3$ the group $\Gamma_0^0(p)$ is no longer equal to $\Gamma(p)$ and we first need to identify  $G_p:=\Gamma^0_0(p)\backslash \PGLZ$ as a subgroup of $\PGLF{p}$.
Then we have to find a subgroup $G$, of some symmetric group ${\S}_{n}$, which is isomorphic to $G_p$. 
The irreducible representations of $G_p$ (for a given, fixed $p$) can then be obtained 
by, for example, a computer algebra program. The biggest problem is clearly to identify fixed vectors under these irreducible constituents with subgroups of $\PSLZ$ and 
possible twists. 

In order to give a flavor  of how the general case might be treated we briefly discuss the cases $\Gamma_0(p^2)$ with  $p=2$, $p=5$ and $p=7$. 
It is known that there is a group $\Gamma^{2}$, of level and index
$2$ (cf.~\cite[1.5]{rankin:mod}), which does not contain $\Gamma_{0}(4)$.
The representation-theoretical approach is very simple in this case since $\Gamma_{0}(4)$ is conjugate
to $\Gamma(2)$ and $\Gamma(2)\backslash\PGLZ\simeq\PGLF{2}$.
Furthermore, $\PGLF{2}$ is isomorphic to the symmetric group $S_{3}$, which only has
three irreducible representations: $\chi_{0}$, $\chi_{\sgn}$ and
$\rho_{\std}$. It is then easy to show that $\chi_0$, $\chi_{\sgn}$ and $\rho_{\std}$ corresponds to $\PSLZ$, $\Gamma^2$, $\Gamma_{0}(2)$, in 
 the same way that the representations
$\chi_0$, $W$ and $\rho_{\std}$ of $\S_{4}$ corresponded to the subgroups $\PSLZ$, $\Gamma^3$ and $\Gamma_0(3)$ in 
Section \ref{sec:A-representation-theoretic-inter}. Using either
this approach, or lemmas analogous to those in Sections \ref{sub:Lifts-into-.Gamma3}
and \ref{sub:Lifts-from-} together with the Selberg trace formula,
it is easy to prove the following theorem.
%%% FS: CHANGED FROM
\begin{thm}
There are no genuinely new forms on $\Gamma_{0}(4)$. 
\end{thm}
Motivated by a brief numerical investigation into the case of $\Gamma_0(25)$ we conjecture 
that there are no genuinely new Maass forms on $\Gamma_0(25)$. 
In the decomposition of $\mathcal{M}(25)$ there are clearly contributions from the oldforms and twists 
from $\Gamma_0(1)$ and $\Gamma_0(5)$ and lifts from the cycloidal group $\Gamma^5$ of index end level $5$. There is also an additional contribution from the 
Maass forms on $\Gamma_0(5)$ with Nebentypus $\chi_{5}=\left(\frac{\cdot}{5}\right)$, twisted by a character of conductor $5$ and order $4$. 
Numerically, we see that each form from $\Gamma^5$ appear on $\Gamma_0(25)$ with multiplicity two, meaning that there should exist a map different from 
$\A_{5}$, taking Maass forms from $\Gamma^5$ to $\Gamma_0(25)$. Under the \emph{assumption} that this is true, and additionally, \emph{assuming} that all involved spaces are orthogonal, it is possible to use 
classical dimension formulas to show that the space of genuinely new holomorphic modular forms of even weight on $\Gamma_0(25)$ is empty. 

Already in the case of $p=7$ the situation becomes much more involved. 
By Cummins and Pauli \cite{MR2016709} we know there are three conjugacy classes of groups between $\PSLZ$ and $\Gamma_0^0(7)$. 
As representatives we can choose  $\Gamma^7$ (the cycloidal group), $7C^{0}$ (of index $14$ and with two cusps) and  $7F^0$ (of index $28$ and with four cusps). 
A numerical investigation of the spectrum of $\Gamma_0(49)$ leaves a few eigenvalues 
 which do not seem to be related to either of the groups $\PSLZ$, $\Gamma_0(7)$, $\Gamma^7$, $7C^{0}$ or $7F^{0}$. 
We calculated $23$ eigenvalues (counted with multiplicity) and out of these there were two eigenvalues which we could not explain using the groups mentioned above
together with Dirichlet characters. 
Both of these eigenvalues seem to appear with multiplicity at least two, and we would therefore not expect them to be genuinely new.
However, the only type of lifts which we have not investigated are those which are associated to 
spaces of Maass forms on the groups $\Gamma^7$, $7C^{0}$ and $7F^0$ together with non-trivial characters. 
We have not determined whether there indeed exists non-trivial
characters on these groups or not. The numerical evidence is therefore
considered to be incomplete and we do not make any conjecture in this
case. Note that the two unexplained eigenvalues were also verified by D. Farmer and S. Lemurell (in connection with
\cite{farmer-lemurell}). 

We leave the matter of a more comprehensive numerical and theoretical study of genuinely new forms on the groups $\Gamma_0(N^2)$ as an open problem.
\subsection*{Acknowledgments}
%\thanks{\label{ackref}
I would like to thank A. Str\"ombergsson for showing me how
to prove the main theorem using the test function $h_{T}$ instead
of evaluating all hyperbolic terms explicitly. I am also grateful
to C. S. Rajan for explaining how the results of \cite{MR1707005} and \cite{MR1994478} transfer to
the setting of Maass waveforms and to D. Farmer, S. Koutsoliotas and S. Ehlen
for many useful comments on previous versions of this paper. 
I would also like to thank K. Buzzard for helping me to understand the automorphic approach in Section \ref{sec:A-representation-theoretic-inter}.
%}

\def\cprime{$'$}

\end{document}